\documentclass[final]{amsart}

\usepackage[utf8]{inputenc}
\usepackage[T1]{fontenc}
\usepackage{amsmath,amssymb,amsthm}
\usepackage{graphicx}
\usepackage{csquotes}
\usepackage{multicol}
\usepackage[svgnames]{xcolor}
\usepackage[inline]{enumitem}
\usepackage{hyperref}
\usepackage{url}
\usepackage{doi}
\usepackage{booktabs}
\usepackage{microtype}
\usepackage{tikz} 
\usepackage[square,numbers]{natbib} 
\usepackage[capitalize,nosort]{cleveref}
\usepackage[notcite,notref,color]{showkeys}

\setlist[enumerate]{label=\textup{(\roman{*})}}

\newcommand{\wstack}[3]{%
    \begingroup%
    \renewcommand*{\arraystretch}{0.8}%
    $\begin{matrix}\ltr{#1}\\\ltr{#2}\\\ltr{#3}\end{matrix}$%
    \endgroup%
}

\usetikzlibrary{arrows,positioning,cd,calc} 
\tikzset{font=\small}
\tikzset{cutting/.style={draw,fill=white,circle,minimum size=3pt,inner sep=0pt}}
\tikzset{word/.style={minimum size=10pt,execute at end node={\strut}}}
\tikzset{morphism/.style={-to,shorten >=3pt,shorten <=3pt}}
\tikzset{state/.style={draw,circle,inner sep=2pt}}
\tikzset{edgelabel/.style={font=\footnotesize}}

\newcommand*{\nn}{\mathbb{N}}
\newcommand*{\ret}{\mathcal{R}}
\newcommand*{\emptyw}{\varepsilon}
\newcommand*{\from}{\colon}

\newcommand{\ltr}[1]{\mathtt{#1}}
\newcommand*{\lang}{\mathcal{L}}
\newcommand*{\lesp}{L}
\newcommand*{\risp}{R}

\newcommand*{\bs}[1]{\mathbf{#1}}
\renewcommand*{\bar}{\overline}

\newcommand*{\epi}{\operatorname{Epi}}
\newcommand*{\std}{\operatorname{Std}}
\newcommand*{\en}{\operatorname{End}}
\newcommand*{\perm}{\operatorname{Perm}}
\newcommand*{\pal}{\operatorname{Pal}}
\newcommand*{\gcp}{\operatorname{gcp}}
\newcommand*{\gcs}{\operatorname{gcs}}
\newcommand*{\ind}{\operatorname{ind}}
\newcommand*{\dee}{\operatorname{d}}

\theoremstyle{plain}
\newtheorem{theorem}{Theorem}[section]
\newtheorem{proposition}[theorem]{Proposition}
\newtheorem{lemma}[theorem]{Lemma}
\newtheorem{corollary}[theorem]{Corollary}

\theoremstyle{definition}
\newtheorem{definition}[theorem]{Definition}
\newtheorem{notation}[theorem]{Notation}
\theoremstyle{remark} 
\newtheorem{remark}[theorem]{Remark}
\newtheorem{example}[theorem]{Example}

\begin{document}

\title{Obstructions to return preservation for episturmian morphisms}

\author[V. Berthé]{Valérie Berthé}

\email{berthe@irif.fr}

\address{IRIF, Université Paris Cité, 75013 Paris, France}

\author[H. Goulet-Ouellet]{Herman Goulet-Ouellet}

\email{herman.goulet.ouellet@fit.cvut.cz}

\address{Czech Technical University in Prague, Faculty of Information Technology}

\begin{abstract}
    This paper studies obstructions to preservation of return sets by episturmian morphisms. We show, by way of an explicit construction, that infinitely many obstructions exist. This generalizes and improves an earlier result about Sturmian morphisms.
\end{abstract}


\thanks{This work was supported by the Agence Nationale de la Recherche through the project SymDynAr (ANR-23-CE40-0024-01). The second author was supported by the CTU Global Postdoc Fellowship program.}

\maketitle

\section{Introduction}

Return words have widely  proved  their importance in word combinatorics and symbolic dynamics. Given a uniformly recurrent infinite word $x$ over a finite alphabet  (uniform recurrence means that factors occur with bounded gaps within $x$), a \emph{return word} is a word which separates successive occurrences of factors in $x$. They are particularly useful when considering infinite words generated by iterating morphisms of the free monoid. As an example, let us quote the characterization of morphic words by Durand~\cite{Durand1998} in terms of so-called \emph{derived sequences} -- sequences obtained by recoding an infinite word with respect to some  set of return words. Durand showed that having a finite number of derived sequences precisely characterizes primitive substitutive sequences.

With this general motivation in mind, we consider  the question of whether or not a given primitive morphism preserves its  return sets, i.e. its sets of return words (with return words  being considered with respect to the infinite word generated by the morphism). Note that the primitivity assumption guarantees that the corresponding infinite words are uniformly recurrent, and thus all return sets are finite. The definition of return sets and the precise formulation of the preservation property are given in \cref{s:preliminaries}.

This paper is a continuation of \cite{Berthe2023}, where the preservation property is introduced and the following results are established. 
\begin{enumerate}[label={\arabic{*}.},wide]
    \item The preservation property holds for every primitive aperiodic bifix morphism, wherein images of letters form a bifix code, for all but finitely many words (\cite{Berthe2023}, Theorem~1).  
    \item This result is applied to the case of the Thue--Morse substitution (which is bifix) in order to describe the subgroups generated by the return sets (\cite{Berthe2023}, Proposition~6). These subgroups form several strictly decreasing chains of subgroups, which contrasts with the dendric case discussed below.  
    \item On the other hand, the preservation property fails for primitive Sturmian morphisms up to conjugacy. More precisely, every conjugacy class of a  primitive Sturmian morphism contains a member for which the property fails infinitely often (\cite{Berthe2023}, Theorem~2).
\end{enumerate}

We focus in the present paper on  the  study of the return preservation property for \emph{episturmian morphisms}, a generalization of Sturmian morphisms. The notion of episturmian word was introduced by Droubay, Justin and Pirillo as a generalization of Sturmian and Arnoux--Rauzy words~\cite{Droubay2001}.  An \emph{episturmian word} is an infinite word  whose language  is closed under  mirror  image  and which admits at most one  factor of every length  having at least two extensions; see Section ~\ref{s:epi} for a precise definition, and  the survey~\cite{Glen2009} for more on episturmian words.

Inspired by analogies with the already existing rich theory of Sturmian morphisms (see e.g. \cite{book/Lothaire2002,book/Fogg2002}), episturmian morphisms have attracted considerable attention~\cite{Glen2009a,Justin2005,Richomme2003,Richomme2003a,Seebold1998}. In particular, they belong to the so-called family of \emph{dendric words}, which have the striking property that their return sets form bases for the free group over the alphabet~\cite{Berthe2015}. This result is  known as the \emph{Return Theorem}  and holds even for dendric words that are not generated by morphisms~\cite{Berthe2015}.  Dendric words are a common generalization of episturmian words and codings of interval exchanges. This kind of stable behavior for subgroups generated by return sets was also observed in the more general class of suffix-connected shift spaces, introduced by the second author~\cite{GouletOuellet2022}. Moreover, the behavior of return sets is closely related to a profinite invariant of shift spaces called the \emph{Sch\"utzenberger group}, as highlighted by Almeida and Costa~\cite{Almeida2016}.

These algebraic considerations provide a further motivation to the present work. Indeed, when it holds, the preservation property provides a simple way to construct infinite chains of return groups (i.e. subgroups of the free group generated by return sets) which can be used to understand the evolution of these groups as the length of the factors goes to infinity. In some cases (for instance for the Thue-Morse morphism), this leads to a very detailed understanding of the general behavior of these subgroups. 

It is with the general goal of understanding the precise limitations of methods based on the preservation property that we set out to identify large families for which the property fails or succeeds, with the present paper being focused on failure. Understanding the failure of the preservation property also allows one to take a closer look at the behaviour of return words, and in particular at \emph{shattering} phenomena (see the discussion at the end of \cref{sec:conclu}). 

We now give some details about our main result. In essence, it states that every primitive episturmian morphism $\sigma$ fails the preservation property for infinitely many words. In its more precise form stated below (\cref{t:main}), we explicitly give a sequence of words which fail the property. We recover, in the case of Sturmian morphisms, the second main result from \cite{Berthe2023}, and in fact improve it by showing how to handle the full conjugacy class. 

Episturmian morphisms are sorted into \emph{conjugacy classes} which have a strong structure, such as highlighted in~\cite{Richomme2003}. This explains the important role played by the integer $\ind(\sigma)$, called the \emph{conjugacy index} of $\sigma$, for the description of words that fail the preservation property; its is  introduced in \cref{s:conjugacy}. This index measures the distance to the unique so-called standard episturmian morphism in its class. In other words, the conjugacy index allows one to locate in a given conjugacy class the episturmian morphism under study. The standard episturmian morphisms also play a special role; they generate words such that all their prefixes are \emph{left special} (i.e. have several left extensions).

Next is the precise statement of our main result. Given a primitive episturmian morphism $\sigma$, we let $\lesp_n$, $\risp_n$ be its unique \emph{left and right special factors} of length $n$; these are the factors that have several extensions in the language of $\sigma$. We also denote by $\Vert\sigma\Vert$ is the sum of the lengths of letter images under $\sigma$.

\begin{theorem}\label{t:main}
    Let $\sigma$ be a primitive episturmian morphism over the alphabet $A$. Then there is a letter $a_{\min} $ in $A$  such that $\sigma$ fails the return preservation property for all but finitely many words in:
    \begin{enumerate}
        \item $\{ a_{\min}\lesp_n \mid \lesp_n=\risp_n \}$ if $\ind(\sigma)\geq|\sigma(a_{\min})|$;
            \label{i:main-1}
        \item $\{\risp_na_{\min} \mid \lesp_n=\risp_n\}$ if $\ind(\sigma)\leq\frac{\Vert\sigma\Vert-|A|}{|A|-1} - |\sigma(a_{\min})|$.
            \label{i:main-2}
    \end{enumerate}
\end{theorem}

For context, the quantity $\frac{\Vert\sigma\Vert-1}{|A|-1}=1+\frac{\Vert\sigma\Vert - |A|}{|A|-1}$ is the cardinality of the conjugacy class of $\sigma$~\cite{Richomme2003}. The proof of the main result relies on a careful study of return sets and Rauzy graphs (see \cref{s:Rauzy}); see in particular~\cref{t:return} where an explicit description of return words is given. The letter $a_{\min}$ appearing in the theorem is called the \emph{minimal letter} and is introduced in \cref{p:minletter-epi}. As highlighted in \cref{c:min}, the inequality $|\sigma(a_{\min})| \leq \frac{\Vert\sigma\Vert - |A|}{|A|-1}-|\sigma(a_{\min})|$ holds, thus implying that every episturmian substitution falls into one of the two cases of \cref{t:main}.

Let us sketch the content of this paper. Basic definitions, and in particular the preservation property, are recalled in \cref{s:preliminaries}. \cref{s:epi} is devoted to episturmian and standard morphisms. \cref{s:conjugacy} focuses on properties of their conjugacy classes, and Rauzy graphs are considered in \cref{s:Rauzy}. Lastly, the proof of \cref{t:main} is given in \cref{s:proofmain} and a  brief conclusion is provided in \cref{sec:conclu}.

\subsection*{Acknowledgements}
 The authors  warmly thank the referees who helped us to  improve the previous version of this paper.

\section{Preliminaries}
\label{s:preliminaries}

Let $A$ be a finite alphabet and $A^*$ be the set of finite words on $A$ equipped with concatenation. We denote the empty word by $\emptyw$ and we let $A^+ = A^*\setminus\{\emptyw\}$.

Let $A^\nn$ be the set of (right-sided) infinite word over $A$. Recall that a (one-sided) \emph{shift space} is a subset $X\subseteq A^\nn$ invariant under the shift map $(x_n)_{n\in\nn}\mapsto (x_{n+1})_{n\in\nn}$ and closed for the product topology of $A^\nn$ ($A$ being discrete). The \emph{language of $X$}, denoted $\lang(X)$, is the set of all finite words occurring as factors in the elements $x\in X$, i.e.\ words of the form 
\begin{equation*}
    x_{[i,j)} = x_i\cdots x_{j-1}\quad \text{for}\ i\leq j\in\nn. 
\end{equation*}

A \emph{left extension} of $v$ in $X$ is a word $u\in A^*$ such that $uv\in \lang(X)$. Likewise, a \emph{right extension} of $v$ is a word $w\in A^*$ such that $vw\in\lang(X)$. We say that $v$ is \emph{left special} if it has at least two left extensions of length 1, and \emph{right special} if those are instead right extensions. A word is \emph{bispecial} if it is simultaneously left and right special. 

Given $u\in \lang(X)$, the following sets are known respectively as the \emph{left} and \emph{right return sets} of $u$ in $X$:
\begin{gather*}
    \ret_X(u) = \{ r\in A^* \mid ru\in \lang(X)\cap uA^+\setminus A^+uA^+\},\\
    \bar\ret_X(u) = \{ r\in A^* \mid ur\in \lang(X)\cap A^+u\setminus A^+uA^+\}.
\end{gather*}
Their elements are called \emph{left} and \emph{right return words} respectively. Observe that these two return sets are related by the following formula (\cite{Almeida2013}, p.~16):
\begin{equation}
    \ret_X(u) = u\bar\ret_X(u)u^{-1}.
    \label{eq:return}
\end{equation}

Let $\sigma\from A^*\to A^*$ be a  morphism. We say that an infinite word $\bs x\in A^\nn$ is a \emph{periodic point} for $\sigma$ if there is an integer $k\geq 1$ such that $\sigma^k(\bs x) = \bs x$. The least such $k$ is known as the \emph{period} of $\bs x$ under $\sigma$. When the period is 1 we say that $\bs x$ is a \emph{fixed point}. Any primitive morphism admits at least one periodic point.

Assume that $\sigma$ is primitive and let $\bs x$ be a periodic point of $\sigma$. We denote by $X_\sigma$ the shift space of $A^\nn$ associated with $\sigma$, which is defined as the closure of the orbit of $\bs x$ under the action of the shift map. This set does not depend on the choice of a periodic point when $\sigma$ is primitive; see for instance~\cite{book/Queffelec2010}, \S~5 or~\cite{book/Durand2022}, \S~1.4. We also let $\lang(\sigma)$ be the language of $X_\sigma$ and $\ret_\sigma(u)$, $\bar\ret_\sigma(u)$ be the left and right return sets of a word $u\in\lang(\sigma)$. Under the primitivity assumption, every periodic point of $\sigma$ is uniformly recurrent, hence all the return sets $\ret_\sigma(u)$ are finite and $X_\sigma$ is minimal; see e.g.~\cite{Durand1998}.

Let $\sigma$ be a primitive morphism over $A$. Given a word $u\in\lang(\sigma)$, consider the following dual preservation properties:
\begin{gather}
    \sigma(\ret_\sigma(u)) = \ret_\sigma(\sigma(u)), \tag{P}\label{P} \\
    \sigma(\bar\ret_\sigma(u)) = \bar\ret_\sigma(\sigma(u)).\tag{P\textquotesingle}\label{P'}
\end{gather}

Because of \eqref{eq:return}, these two properties turn out to be equivalent (\cite{Berthe2023}, Lemma~3). The property \eqref{P} is  the main actor of the present paper. 

Let $\sigma$ and $\rho$ be two morphisms. We say that $\rho$ is a \emph{right conjugate} of $\sigma$, or $\sigma$ a \emph{left conjugate} of $\rho$, if there is a word $w\in A^*$ such that $\sigma(a)w = w\rho(a)$ for all $a\in A$; this is written $\sigma = w\rho w^{-1}$ or $\rho = w^{-1}\sigma w$. We say that two morphisms are conjugate if the first is a right conjugate of the second or vice-versa. This is an equivalence relation~(\cite{Richomme2003}, Lemma~3.1), so we may speak of \emph{conjugacy classes}.

\section{First properties of episturmian morphisms}
\label{s:epi}

In this section, we  gather some basic properties of episturmian morphisms. For a more detailed account on the theory of episturmian words, see the survey~\cite{Glen2009}. 
Let us first recall some key definitions. 

\begin{enumerate}[label={\arabic{*}.},wide]
    \item An \emph{episturmian} word is an infinite word $\bs x\in A^\nn$ whose language $\lang(\bs x)$ is closed under reversal (mirror) and which admits at most one left special factor of every length (this is \cite{Droubay2001}, Theorem~5, used here as a definition). 
    \item An episturmian word is called \emph{non-degenerate} if it admits one left special factor of every length.
    \item A \emph{standard episturmian} word is a non-degenerate episturmian word $\bs x$ whose left special factors occur as prefixes, that is they are all of the form $\bs x_{[0,n)}$, $n\in\nn$. 
    \item A \emph{strict episturmian} word is a non-degenerate episturmian word whose left special factors all have $|A|$ left extensions.
\end{enumerate}

Strict episturmian words are sometimes called \emph{Arnoux--Rauzy} words. Because of closure under mirror, left special factors and left extensions can equivalently be replaced by their right counterparts in the above definitions. Episturmian words are \emph{uniformly recurrent} (\cite{Droubay2001}, Theorem~2), hence they generate minimal shift spaces. A minimal shift space containing \emph{one} episturmian word consists therefore \emph{only} of episturmian words; we call those \emph{episturmian shift spaces}. Likewise it makes sense to speak of \emph{strict episturmian shift spaces}.

A morphism $\sigma\from A^*\to A^*$ is called \emph{episturmian} if it preserves episturmian words. A \emph{standard morphism} is an episturmian morphism which additionally preserves standard episturmian words. Let $\epi(A)$ be the set of all episturmian morphisms of $A^*$ and $\std(A)$ be the set of all standard morphisms of $A^*$. Note that $\std(A)$ is a submonoid of $\epi(A)$ which is itself a submonoid of $\en(A^*)$, the monoid of all endomorphisms of $A^*$. 

For $a\in A$, consider the morphisms $\psi_a,\bar\psi_a\from A^*\to A^*$ defined by
\begin{equation*}
    \psi_a(b) = 
    \begin{cases}
        a & \text{if }b=a,\\
        ab & \text{if }b\neq a;
    \end{cases}\qquad 
    \bar\psi_a(b) = 
    \begin{cases}
        a & \text{if }b=a,\\
        ba & \text{if }b\neq a.
    \end{cases}
\end{equation*}

In other words, on an alphabet $A = \{\ltr{a},\ltr{b},\ltr{c},\cdots\}$, one has 
\begin{equation*}
    \psi_{\ltr{a}}(\ltr{a}) = \ltr{a},\quad \psi_{\ltr{a}}(\ltr{b})=\ltr{ab},\quad \psi_{\ltr{a}}(\ltr{c}) = \ltr{ac},\quad\cdots.
\end{equation*}

Let $\perm(A)$ be the group of permutations of $A$, which we also view as automorphisms of $A^*$.
\begin{example}[$d$-bonacci]
    \label{e:dbo}
    Let $d\geq 2$ and $A = \{\ltr{a}_1,\cdots,\ltr{a}_d\}$ be a $d$-letter alphabet. Let $\sigma = \psi_{\ltr{a}_1}\circ\pi$, where $\pi = (\ltr{a}_1 \ \ltr{a}_2 \ \cdots \ \ltr{a}_d)$, with the usual cycle notation for permutations. This is a standard episturmian morphism known as the \emph{$d$-bonacci morphism}. When $d=2$ we get the Fibonacci morphism
    \begin{equation*}
        \sigma(\ltr{a}) = \ltr{a}\ltr{b},\quad \sigma(\ltr{a}) = \ltr{a},
    \end{equation*}
    and with $d=3$ the \enquote{tribonacci} morphism, first studied by Rauzy in~\cite{Rauzy1982},
    \begin{equation*}
        \sigma(\ltr{a}) = \ltr{ab},\quad \sigma(\ltr{b}) = \ltr{ac},\quad \sigma(\ltr{c}) = \ltr{a}.
    \end{equation*}
    Observe that $\sigma^d = \psi_{\ltr{a}_1}\circ\psi_{\ltr{a}_2}\circ\cdots\circ\psi_{\ltr{a}_d}$.
\end{example}

Let us mention the following important result of Justin and Pirillo.
\begin{theorem}[\cite{Justin2002}]
    The monoid $\epi(A)$ is generated by the endomorphisms $\psi_a$, $\bar\psi_a$ for $a\in A$ together with the elements of $\perm(A)$. 

    The monoid $\std(A)$ is generated by the endomorphisms $\psi_a$ for $a\in A$ together with the elements of $\perm(A)$.
\end{theorem}

Let $\bar A = \{\bar a \mid a\in A\}$ be a disjoint copy of $A$ and extend $\psi$  as a map on $(A\cup\bar A)$ to a map on $(A\cup\bar A)^*$ defined by $\psi_{\bar a} = \bar\psi_a$ and $\psi_{uv} = \psi_u\circ\psi_v$. Note that for all $u\in(A\cup\bar A)^*$ and $\pi\in\perm(A)$
\begin{equation}\label{eq:commute}
    \psi_{\pi(u)}\circ\pi = \pi\circ\psi_{u}.
\end{equation}
In particular, the monoids $\epi(A)$ and $\std(A)$ admit the following descriptions:
\begin{gather*}
    \epi(A) = \{\psi_v\circ\pi \mid v\in(A\cup\bar A)^*, \pi\in\perm(A)\},\\
    \std(A) = \{\psi_u\circ\pi\mid u\in A^*, \pi\in\perm(A)\}.
\end{gather*}

Thus the tuples of the form $(\psi_u(a))_{a\in A}$ with $u\in A^*$ represent all standard episturmian morphisms up to permutation. These tuples, called \emph{standard tuples}, form the \emph{standard tree}; the standard ternary tree is depicted in \cref{f:std-tree}. 

\begin{figure}
\begin{tikzpicture}[font=\scriptsize,every node/.style={inner sep=2pt}]
        \node (e) at (0,0)   {\wstack{a}{b}{c}};
            \node (a) at (90:2)  {\wstack{a}{ab}{ac}};
                \node (aa) at ([shift=({90:2})]a) {\wstack{a}{aab}{aac}};
                    \node (aaa) at ([shift=({90:2})]aa)  {\wstack{a}{aaab}{aaac}};
                    \node (aab) at ([shift=({120:2})]aa) {\wstack{aaba}{aab}{aabaac}};
                    \node (aac) at ([shift=({60:2})]aa)  {\wstack{aaca}{aacaab}{aac}};
                \node (ab) at ([shift=({150:2})]a) {\wstack{aba}{ab}{abac}};
                    \node (aba) at ([shift=({150:2})]ab) {\wstack{aba}{abaab}{abaabac}};
                    \node (abb) at ([shift=({180:2})]ab) {\wstack{ababa}{ab}{ababac}};
                    \node (abc) at ([shift=({120:2})]ab) {\wstack{abacaba}{abacab}{abac}};
                \node (ac) at ([shift=({30:2})]a) {\wstack{aca}{acab}{ac}};
                    \node (aca) at ([shift=({30:2})]ac) {\wstack{aca}{acaacab}{acaac}};
                    \node (acb) at ([shift=({60:2})]ac) {\wstack{acabaca}{acab}{acabac}};
                    \node (acc) at ([shift=({0: 2})]ac) {\wstack{acaca}{acacab}{ac}};
            \node (b) at (330:2) {\wstack{ba}{b}{bc}};
                \node (ba) at ([shift=({330:2})]b) {\wstack{ba}{bab}{babc}};
                    \node (baa) at ([shift=({330:2})]ba) {\wstack{ba}{babab}{bababc}};
                    \node (bab) at ([shift=({360:2})]ba) {\wstack{babba}{bab}{babbabc}};
                    \node (bac) at ([shift=({300:2})]ba) {\wstack{babcba}{babcbab}{babc}};
                \node (bb) at ([shift=({390:2})]b) {\wstack{bba}{b}{bbc}};
                    \node (bba) at ([shift=({390:2})]bb) {\wstack{bba}{bbab}{bbabbc}};
                    \node (bbb) at ([shift=({420:2})]bb) {\wstack{bbba}{b}{bbbc}};
                    \node (bbc) at ([shift=({360:2})]bb) {\wstack{bbcbba}{bbcb}{bbc}};
                \node (bc) at ([shift=({270:2})]b) {\wstack{bcba}{bcb}{bc}};
                    \node (bca) at ([shift=({270:2})]bc) {\wstack{bcba}{bcbabcb}{bcbabc}};
                    \node (bcb) at ([shift=({300:2})]bc) {\wstack{bcbbcba}{bcb}{bcbbc}};
                    \node (bcc) at ([shift=({240:2})]bc) {\wstack{bcbcba}{bcbcb}{bc}};
            \node (c) at (210:2) {\wstack{ca}{cb}{c}};
                \node (ca) at ([shift=({210:2})]c) {\wstack{ca}{cacb}{cac}};
                    \node (caa) at ([shift=({210:2})]ca) {\wstack{ca}{cacacb}{cacac}};
                    \node (cab) at ([shift=({240:2})]ca) {\wstack{cacbca}{cacb}{cacbcac}};
                    \node (cac) at ([shift=({180:2})]ca) {\wstack{cacca}{caccacb}{cac}};
                \node (cb) at ([shift=({270:2})]c) {\wstack{cbca}{cb}{cbc}};
                    \node (cba) at ([shift=({270:2})]cb) {\wstack{cbca}{cbcacb}{cbcacbc}};
                    \node (cbb) at ([shift=({300:2})]cb) {\wstack{cbcbca}{cb}{cbcbc}};
                    \node (cbc) at ([shift=({240:2})]cb) {\wstack{cbccbca}{cbccb}{cbc}};
                \node (cc) at ([shift=({150:2})]c) {\wstack{cca}{ccb}{c}};
                    \node (cca) at ([shift=({150:2})]cc) {\wstack{cca}{ccaccb}{ccac}};
                    \node (ccb) at ([shift=({180:2})]cc) {\wstack{ccbcca}{ccb}{ccbc}};
                    \node (ccc) at ([shift=({120:2})]cc) {\wstack{ccca}{cccb}{c}};
        \draw[-{to}] (e) to node  [draw,fill=white,circle,font=\tiny,inner sep=1pt] {$\ltr{a}$} (a);
        \draw[-{to}] (e) to node  [draw,fill=white,circle,font=\tiny,inner sep=1pt] {$\ltr{b}$} (b);
        \draw[-{to}] (e) to node  [draw,fill=white,circle,font=\tiny,inner sep=1pt] {$\ltr{c}$} (c);
        \draw[-{to}] (a) to node  [draw,fill=white,circle,font=\tiny,inner sep=1pt] {$\ltr{a}$} (aa);
        \draw[-{to}] (a) to node  [draw,fill=white,circle,font=\tiny,inner sep=1pt] {$\ltr{b}$} (ab);
        \draw[-{to}] (a) to node  [draw,fill=white,circle,font=\tiny,inner sep=1pt] {$\ltr{c}$} (ac);
        \draw[-{to}] (b) to node  [draw,fill=white,circle,font=\tiny,inner sep=1pt] {$\ltr{a}$} (ba);
        \draw[-{to}] (b) to node  [draw,fill=white,circle,font=\tiny,inner sep=1pt] {$\ltr{b}$} (bb);
        \draw[-{to}] (b) to node  [draw,fill=white,circle,font=\tiny,inner sep=1pt] {$\ltr{c}$} (bc);
        \draw[-{to}] (c) to node  [draw,fill=white,circle,font=\tiny,inner sep=1pt] {$\ltr{a}$} (ca);
        \draw[-{to}] (c) to node  [draw,fill=white,circle,font=\tiny,inner sep=1pt] {$\ltr{b}$} (cb);
        \draw[-{to}] (c) to node  [draw,fill=white,circle,font=\tiny,inner sep=1pt] {$\ltr{c}$} (cc);
        \draw[-{to}] (aa) to node [draw,fill=white,circle,font=\tiny,inner sep=1pt] {$\ltr{a}$} (aaa);
        \draw[-{to}] (aa) to node [draw,fill=white,circle,font=\tiny,inner sep=1pt] {$\ltr{b}$} (aab);
        \draw[-{to}] (aa) to node [draw,fill=white,circle,font=\tiny,inner sep=1pt] {$\ltr{c}$} (aac);
        \draw[-{to}] (ab) to node [draw,fill=white,circle,font=\tiny,inner sep=1pt] {$\ltr{a}$} (aba);
        \draw[-{to}] (ab) to node [draw,fill=white,circle,font=\tiny,inner sep=1pt] {$\ltr{b}$} (abb);
        \draw[-{to}] (ab) to node [draw,fill=white,circle,font=\tiny,inner sep=1pt] {$\ltr{c}$} (abc);
        \draw[-{to}] (ac) to node [draw,fill=white,circle,font=\tiny,inner sep=1pt] {$\ltr{a}$} (aca);
        \draw[-{to}] (ac) to node [draw,fill=white,circle,font=\tiny,inner sep=1pt] {$\ltr{b}$} (acb);
        \draw[-{to}] (ac) to node [draw,fill=white,circle,font=\tiny,inner sep=1pt] {$\ltr{c}$} (acc);
        \draw[-{to}] (ba) to node [draw,fill=white,circle,font=\tiny,inner sep=1pt] {$\ltr{a}$} (baa);
        \draw[-{to}] (ba) to node [draw,fill=white,circle,font=\tiny,inner sep=1pt] {$\ltr{b}$} (bab);
        \draw[-{to}] (ba) to node [draw,fill=white,circle,font=\tiny,inner sep=1pt] {$\ltr{c}$} (bac);
        \draw[-{to}] (bb) to node [draw,fill=white,circle,font=\tiny,inner sep=1pt] {$\ltr{a}$} (bba);
        \draw[-{to}] (bb) to node [draw,fill=white,circle,font=\tiny,inner sep=1pt] {$\ltr{b}$} (bbb);
        \draw[-{to}] (bb) to node [draw,fill=white,circle,font=\tiny,inner sep=1pt] {$\ltr{c}$} (bbc);
        \draw[-{to}] (bc) to node [draw,fill=white,circle,font=\tiny,inner sep=1pt] {$\ltr{a}$} (bca);
        \draw[-{to}] (bc) to node [draw,fill=white,circle,font=\tiny,inner sep=1pt] {$\ltr{b}$} (bcb);
        \draw[-{to}] (bc) to node [draw,fill=white,circle,font=\tiny,inner sep=1pt] {$\ltr{c}$} (bcc);
        \draw[-{to}] (ca) to node [draw,fill=white,circle,font=\tiny,inner sep=1pt] {$\ltr{a}$} (caa);
        \draw[-{to}] (ca) to node [draw,fill=white,circle,font=\tiny,inner sep=1pt] {$\ltr{b}$} (cab);
        \draw[-{to}] (ca) to node [draw,fill=white,circle,font=\tiny,inner sep=1pt] {$\ltr{c}$} (cac);
        \draw[-{to}] (cb) to node [draw,fill=white,circle,font=\tiny,inner sep=1pt] {$\ltr{a}$} (cba);
        \draw[-{to}] (cb) to node [draw,fill=white,circle,font=\tiny,inner sep=1pt] {$\ltr{b}$} (cbb);
        \draw[-{to}] (cb) to node [draw,fill=white,circle,font=\tiny,inner sep=1pt] {$\ltr{c}$} (cbc);
        \draw[-{to}] (cc) to node [draw,fill=white,circle,font=\tiny,inner sep=1pt] {$\ltr{a}$} (cca);
        \draw[-{to}] (cc) to node [draw,fill=white,circle,font=\tiny,inner sep=1pt] {$\ltr{b}$} (ccb);
        \draw[-{to}] (cc) to node [draw,fill=white,circle,font=\tiny,inner sep=1pt] {$\ltr{c}$} (ccc);
    \end{tikzpicture}
    \caption{The standard ternary tree.}
    \label{f:std-tree}
\end{figure}
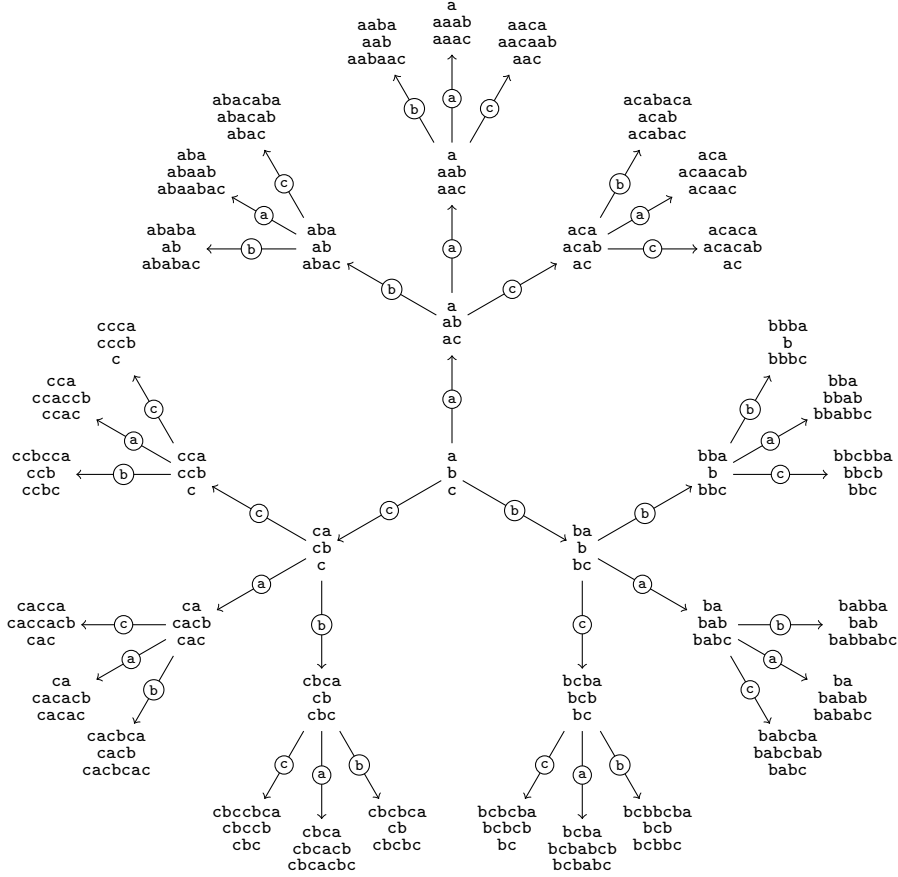

We finish this section with a discussion on periodic points of episturmian morphisms, the main conclusion being that for primitive episturmian morphisms, all of them are strict episturmian words. This is not necessarily the case for non-primitive episturmian morphisms, since for instance permutations admit all infinite words as periodic points. We also have the following less trivial example of this.

\begin{example}\label{ex:contrex}
    Consider $\sigma = \bar\psi_{\ltr{ab}}$ on the alphabet $A = \{\ltr{a},\ltr{b},\ltr{c}\}$. Then $\sigma$ has a fixed point of the form $\bs{x} = \ltr{c}\bs{y}$ where $\bs{y}$ is the Sturmian word determined by the restriction of $\sigma$ to $\{\ltr{a},\ltr{b}\}^*$ (in fact $\bs{y}$ is the Fibonacci word). This fixed point is not uniformly recurrent, hence it cannot be episturmian. Notice also that it is not closed under reversal, although it retains the property of having exactly one left and one right special factor of every length.
\end{example}

\begin{proposition}\label{p:strict}
    Every episturmian morphism $\sigma$ admits an episturmian periodic point. If moreover $\sigma$ is primitive, then all of its periodic points are strict episturmian words and $X_\sigma$ is a strict episturmian shift space.
\end{proposition}

\begin{proof}
    By \eqref{eq:commute}, every episturmian morphism has a power of the form $\psi_w$ for some $w\in (A\cup\bar A)^*$: for an episturmian morphism of the form $\psi_u\circ\pi$ with $u\in (A\cup \bar A)^*$ and $\pi\in\perm(A)$, take $w = u\pi(u)\cdots\pi^{n-1}(u)$ where $n$ is order of $\pi$; then $(\psi_u\circ\pi)^n = \psi_w$. By \cite{Justin2005}, Proposition~3.11, $\psi_w$ acts as a transformation on some finite set of episturmian words, thus it must have a power fixing one of them.
  
    Next we let $\sigma$ be a primitive episturmian morphism. By the first part of the statement, $\sigma$ admits some episturmian periodic point, which generates $X_\sigma$; hence $X_\sigma$ is an episturmian shift space. Since all periodic points of $\sigma$ are contained in $X_\sigma$, they must be episturmian. It remains to show that they are strict.
    
    Up to taking a power, we may assume that $\sigma = \psi_w$ for some word $w\in (A\cup\bar A)^*$. The fact that $\sigma$ is primitive then implies that for every letter $a\in A$, either $a$ or $\bar a$ occurs in $u$. By \cite{Droubay2001}, Theorem~7, the periodic points of $\sigma$ must be strict episturmian words. 
\end{proof}

\section{Conjugacy classes of episurmian morphisms}
\label{s:conjugacy}

We focus in this section on a description of conjugacy classes of episturmian morphisms. We first recall the following foundational result by Richomme; see also~\cite{Richomme2003a} for related algorithms.

\begin{theorem}[\cite{Richomme2003}, Theorem~5.1]\label{t:richomme}
    A morphism is episturmian if and only if it is a right conjugate of a standard morphism, which is moreover unique. 
\end{theorem}

Thanks to this theorem, the following notion is well-defined.
\begin{definition}\label{d:index}
    Let $\sigma\in\epi(A)$ be an episturmian morphism. The \emph{conjugacy index} of $\sigma$, denoted $\ind(\sigma)$, is the length of the unique word $w$ such that $w\sigma w^{-1}\in\std(A)$.
\end{definition}

\cref{tb:conj} gives an example of an episturmian conjugacy class ordered by index.

\begin{table}
    \begin{tabular}{ccc}
        \toprule
        Morphism & Prefix of $\pal$ & Conjugacy index \\
        \midrule
        $\ltr{a}\mapsto\ltr{ababa}, \ltr{b}\mapsto\ltr{ababac}, \ltr{c}\mapsto\ltr{ab}$ & $\emptyw$ & 0 \\
        $\ltr{a}\mapsto\ltr{babaa}, \ltr{b}\mapsto\ltr{babaca}, \ltr{c}\mapsto\ltr{ba}$ & $\ltr{a}$ & 1 \\
        $\ltr{a}\mapsto\ltr{abaab}, \ltr{b}\mapsto\ltr{abacab}, \ltr{c}\mapsto\ltr{ab}$ & $\ltr{ab}$ & 2 \\
        $\ltr{a}\mapsto\ltr{baaba}, \ltr{b}\mapsto\ltr{bacaba}, \ltr{c}\mapsto\ltr{ba}$ & $\ltr{aba}$ & 3 \\
        $\ltr{a}\mapsto\ltr{aabab}, \ltr{b}\mapsto\ltr{acabab}, \ltr{c}\mapsto\ltr{ab}$ & $\ltr{abab}$ & 4 \\
        $\ltr{a}\mapsto\ltr{ababa}, \ltr{b}\mapsto\ltr{cababa}, \ltr{c}\mapsto\ltr{ba}$ & $\ltr{ababa}$ & 5 \\
        \bottomrule\addlinespace
    \end{tabular}
    \caption{An episturmian conjugacy class.}
    \label{tb:conj}
\end{table}

We next proceed to give explicit formulas for the conjugacy index in \cref{c:index} based on \cref{p:rcp}, which relates episturmian conjugacy classes to the following operator. Let $\pal\from A^*\to A^*$ be the \emph{iterated palindromic closure}, defined inductively by 
\begin{equation*}
    \pal(\emptyw) = \emptyw,\quad \pal(ua) = (\pal(u)a)^{(+)}\ \text{for}\ u\in A^*, \ a\in A,
\end{equation*}
where $x^{(+)}$ is the shortest palindrome having $x$ as a prefix. This notion originates from a paper of de Luca \cite{Luca1997a}, where it is denoted $\psi$; see \cite{Perrin2023} for recent work concerning $\pal$. Some of its values on ternary words are given in \cref{f:pal-tree} as illustration. 

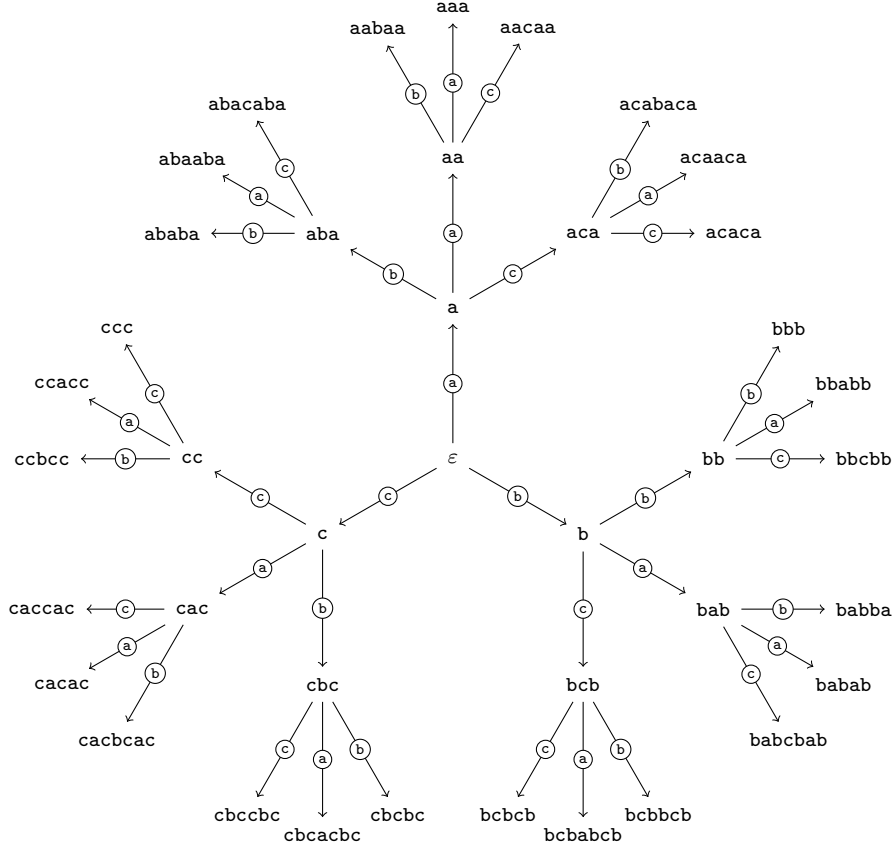
\begin{figure}
    \begin{tikzpicture}[font=\footnotesize,every node/.style={inner sep=4pt}]
        \node (e) at (0,0)   {$\emptyw$};
            \node (a) at (90:2)  {$\ltr{a}$};
                \node (aa) at ([shift=({90:2})]a) {$\ltr{aa}$};
                    \node (aaa) at ([shift=({90:2})]aa)  {$\ltr{aaa}$};
                    \node (aab) at ([shift=({120:2})]aa) {$\ltr{aabaa}$};
                    \node (aac) at ([shift=({60:2})]aa)  {$\ltr{aacaa}$};
                \node (ab) at ([shift=({150:2})]a) {$\ltr{aba}$};
                    \node (aba) at ([shift=({150:2})]ab) {$\ltr{abaaba}$};
                    \node (abb) at ([shift=({180:2})]ab) {$\ltr{ababa}$};
                    \node (abc) at ([shift=({120:2})]ab) {$\ltr{abacaba}$};
                \node (ac) at ([shift=({30:2})]a) {$\ltr{aca}$};
                    \node (aca) at ([shift=({30:2})]ac) {$\ltr{acaaca}$};
                    \node (acb) at ([shift=({60:2})]ac) {$\ltr{acabaca}$};
                    \node (acc) at ([shift=({0: 2})]ac) {$\ltr{acaca}$};
            \node (b) at (330:2) {$\ltr{b}$};
                \node (ba) at ([shift=({330:2})]b) {$\ltr{bab}$};
                    \node (baa) at ([shift=({330:2})]ba) {$\ltr{babab}$};
                    \node (bab) at ([shift=({360:2})]ba) {$\ltr{babba}$};
                    \node (bac) at ([shift=({300:2})]ba) {$\ltr{babcbab}$};
                \node (bb) at ([shift=({390:2})]b) {$\ltr{bb}$};
                    \node (bba) at ([shift=({390:2})]bb) {$\ltr{bbabb}$};
                    \node (bbb) at ([shift=({420:2})]bb) {$\ltr{bbb}$};
                    \node (bbc) at ([shift=({360:2})]bb) {$\ltr{bbcbb}$};
                \node (bc) at ([shift=({270:2})]b) {$\ltr{bcb}$};
                    \node (bca) at ([shift=({270:2})]bc) {$\ltr{bcbabcb}$};
                    \node (bcb) at ([shift=({300:2})]bc) {$\ltr{bcbbcb}$};
                    \node (bcc) at ([shift=({240:2})]bc) {$\ltr{bcbcb}$};
            \node (c) at (210:2) {$\ltr{c}$};
                \node (ca) at ([shift=({210:2})]c) {$\ltr{cac}$};
                    \node (caa) at ([shift=({210:2})]ca) {$\ltr{cacac}$};
                    \node (cab) at ([shift=({240:2})]ca) {$\ltr{cacbcac}$};
                    \node (cac) at ([shift=({180:2})]ca) {$\ltr{caccac}$};
                \node (cb) at ([shift=({270:2})]c) {$\ltr{cbc}$};
                    \node (cba) at ([shift=({270:2})]cb) {$\ltr{cbcacbc}$};
                    \node (cbb) at ([shift=({300:2})]cb) {$\ltr{cbcbc}$};
                    \node (cbc) at ([shift=({240:2})]cb) {$\ltr{cbccbc}$};
                \node (cc) at ([shift=({150:2})]c) {$\ltr{cc}$};
                    \node (cca) at ([shift=({150:2})]cc) {$\ltr{ccacc}$};
                    \node (ccb) at ([shift=({180:2})]cc) {$\ltr{ccbcc}$};
                    \node (ccc) at ([shift=({120:2})]cc) {$\ltr{ccc}$};
        \draw[-{to}] (e) to node  [draw,fill=white,circle,inner sep=1pt,font=\tiny] {$\ltr{a}$} (a);
        \draw[-{to}] (e) to node  [draw,fill=white,circle,inner sep=1pt,font=\tiny] {$\ltr{b}$} (b);
        \draw[-{to}] (e) to node  [draw,fill=white,circle,inner sep=1pt,font=\tiny] {$\ltr{c}$} (c);
        \draw[-{to}] (a) to node  [draw,fill=white,circle,inner sep=1pt,font=\tiny] {$\ltr{a}$} (aa);
        \draw[-{to}] (a) to node  [draw,fill=white,circle,inner sep=1pt,font=\tiny] {$\ltr{b}$} (ab);
        \draw[-{to}] (a) to node  [draw,fill=white,circle,inner sep=1pt,font=\tiny] {$\ltr{c}$} (ac);
        \draw[-{to}] (b) to node  [draw,fill=white,circle,inner sep=1pt,font=\tiny] {$\ltr{a}$} (ba);
        \draw[-{to}] (b) to node  [draw,fill=white,circle,inner sep=1pt,font=\tiny] {$\ltr{b}$} (bb);
        \draw[-{to}] (b) to node  [draw,fill=white,circle,inner sep=1pt,font=\tiny] {$\ltr{c}$} (bc);
        \draw[-{to}] (c) to node  [draw,fill=white,circle,inner sep=1pt,font=\tiny] {$\ltr{a}$} (ca);
        \draw[-{to}] (c) to node  [draw,fill=white,circle,inner sep=1pt,font=\tiny] {$\ltr{b}$} (cb);
        \draw[-{to}] (c) to node  [draw,fill=white,circle,inner sep=1pt,font=\tiny] {$\ltr{c}$} (cc);
        \draw[-{to}] (aa) to node [draw,fill=white,circle,inner sep=1pt,font=\tiny] {$\ltr{a}$} (aaa);
        \draw[-{to}] (aa) to node [draw,fill=white,circle,inner sep=1pt,font=\tiny] {$\ltr{b}$} (aab);
        \draw[-{to}] (aa) to node [draw,fill=white,circle,inner sep=1pt,font=\tiny] {$\ltr{c}$} (aac);
        \draw[-{to}] (ab) to node [draw,fill=white,circle,inner sep=1pt,font=\tiny] {$\ltr{a}$} (aba);
        \draw[-{to}] (ab) to node [draw,fill=white,circle,inner sep=1pt,font=\tiny] {$\ltr{b}$} (abb);
        \draw[-{to}] (ab) to node [draw,fill=white,circle,inner sep=1pt,font=\tiny] {$\ltr{c}$} (abc);
        \draw[-{to}] (ac) to node [draw,fill=white,circle,inner sep=1pt,font=\tiny] {$\ltr{a}$} (aca);
        \draw[-{to}] (ac) to node [draw,fill=white,circle,inner sep=1pt,font=\tiny] {$\ltr{b}$} (acb);
        \draw[-{to}] (ac) to node [draw,fill=white,circle,inner sep=1pt,font=\tiny] {$\ltr{c}$} (acc);
        \draw[-{to}] (ba) to node [draw,fill=white,circle,inner sep=1pt,font=\tiny] {$\ltr{a}$} (baa);
        \draw[-{to}] (ba) to node [draw,fill=white,circle,inner sep=1pt,font=\tiny] {$\ltr{b}$} (bab);
        \draw[-{to}] (ba) to node [draw,fill=white,circle,inner sep=1pt,font=\tiny] {$\ltr{c}$} (bac);
        \draw[-{to}] (bb) to node [draw,fill=white,circle,inner sep=1pt,font=\tiny] {$\ltr{a}$} (bba);
        \draw[-{to}] (bb) to node [draw,fill=white,circle,inner sep=1pt,font=\tiny] {$\ltr{b}$} (bbb);
        \draw[-{to}] (bb) to node [draw,fill=white,circle,inner sep=1pt,font=\tiny] {$\ltr{c}$} (bbc);
        \draw[-{to}] (bc) to node [draw,fill=white,circle,inner sep=1pt,font=\tiny] {$\ltr{a}$} (bca);
        \draw[-{to}] (bc) to node [draw,fill=white,circle,inner sep=1pt,font=\tiny] {$\ltr{b}$} (bcb);
        \draw[-{to}] (bc) to node [draw,fill=white,circle,inner sep=1pt,font=\tiny] {$\ltr{c}$} (bcc);
        \draw[-{to}] (ca) to node [draw,fill=white,circle,inner sep=1pt,font=\tiny] {$\ltr{a}$} (caa);
        \draw[-{to}] (ca) to node [draw,fill=white,circle,inner sep=1pt,font=\tiny] {$\ltr{b}$} (cab);
        \draw[-{to}] (ca) to node [draw,fill=white,circle,inner sep=1pt,font=\tiny] {$\ltr{c}$} (cac);
        \draw[-{to}] (cb) to node [draw,fill=white,circle,inner sep=1pt,font=\tiny] {$\ltr{a}$} (cba);
        \draw[-{to}] (cb) to node [draw,fill=white,circle,inner sep=1pt,font=\tiny] {$\ltr{b}$} (cbb);
        \draw[-{to}] (cb) to node [draw,fill=white,circle,inner sep=1pt,font=\tiny] {$\ltr{c}$} (cbc);
        \draw[-{to}] (cc) to node [draw,fill=white,circle,inner sep=1pt,font=\tiny] {$\ltr{a}$} (cca);
        \draw[-{to}] (cc) to node [draw,fill=white,circle,inner sep=1pt,font=\tiny] {$\ltr{b}$} (ccb);
        \draw[-{to}] (cc) to node [draw,fill=white,circle,inner sep=1pt,font=\tiny] {$\ltr{c}$} (ccc);
    \end{tikzpicture}
    \caption{Values of $\pal$ over a ternary alphabet.}
    \label{f:pal-tree}
\end{figure}

Recall Justin's left and right formulas from~\cite{Justin2005} relating $\pal$ with the morphisms $\psi_a$ and $\bar\psi_a$:
\begin{gather}
    \pal(uv) = \psi_u(\pal(v))\pal(u),\tag{J}\label{J}\\
    \pal(uv) = \pal(u)\bar\psi_u(\pal(v)).\tag{J\textquotesingle}\label{J'}
\end{gather}

Let $\gcp$ denote the \emph{greatest common prefix} operator, i.e.\ $\gcp(u,v)$ is the longest word which is a prefix of both $u$ and $v$. It is related with $\pal$ by the following.
\begin{lemma}\label{l:pal-gcp}
    Let $u\in A^*$ and $a\neq b\in A$. Then $\pal(u)a$ is a prefix of $\psi_u(ab)$ and $\pal(u) = \gcp( \psi_u(ab), \psi_u(ba) )$.
\end{lemma}

\begin{proof}
    The proof is done by induction on $|u|$. If $u=\emptyw$ the statement is trivial. Assume the result holds for $u$ and let $c\in A$. Fix $b\in A$ with $a\neq b$. The rest of the proof has two cases.
    
    Case $a\neq c$. By induction, there is $v\in A^*$ such that $\psi_u(ac) = \pal(u)av$. Let $d = b$ if $c=b$ and $d=\emptyw$ otherwise. Using Justin's formula \eqref{J} we get
        \begin{align*}
            \psi_{uc}(ab) &= \psi_u(cacd) = \psi_u(c)\psi_u(ac)\psi_u(d) \\
                          &= \psi_u(c)\pal(u)av\psi_u(d) = \pal(uc)av\psi_u(d). 
        \end{align*}

    Case $a=c$. By induction, there is $v\in A^*$ such that $\psi_u(ab) = \pal(u)av$ and, using Justin's formula \eqref{J},
        \begin{equation*}
            \psi_{ua}(ab) = \psi_u(aab) = \psi_u(a)\psi_u(ab) = \psi_u(a)\pal(u)av = \pal(ua)av. 
        \end{equation*}

The second part of the claim follows immediately.
\end{proof}

For a morphism $\sigma\from A^*\to A^*$, we let $\Vert\sigma\Vert=\sum_{a\in A}|\sigma(a)|$. In \cite{Justin2005}, Corollary~2.7, Justin derived the following formula as a consequence of \eqref{J}:
\begin{equation}
    \label{length-pal}
    |\pal(u)|+1 = \frac{\Vert \psi_u\Vert - 1}{|A|-1}.
\end{equation}

As a case study, consider the binary alphabet $A=\{\ltr{a},\ltr{b}\}$. Then the above formula simplifies to $|\pal(u)| = |\psi_u(\ltr{ab})|-2$. For instance, when $u = \ltr{aba}$, we have $\psi_u(\ltr{a}) = \ltr{aba}$, $\psi_u(\ltr{b}) = \ltr{abaab}$, so the formula yields $|\pal(u)| = 6$. And indeed, one checks that $\pal(u) = \ltr{abaaba}$.

Comparing the equality $|\pal(u)| = |\psi_u(\ltr{ab})|-2$ with \cref{l:pal-gcp} gives a proof of the fact, dating back to a paper of Séébold, that for every standard Sturmian morphism $\sigma$, the two words $\sigma(\ltr{ab})$ and $\sigma(\ltr{ba})$ agree except for their two last letters (\cite{Seebold1998}, Property~4).

In general, the right-hand side of \eqref{length-pal} is also, according to Richomme, the formula for the size of conjugacy classes of episturmian morphisms.

\begin{proposition}[\cite{Richomme2003}, Property~4.1]
    Let $\sigma$ be an episturmian morphism. The cardinality of the conjugacy class of $\sigma$ equals $\frac{\Vert\sigma\Vert-1}{|A|-1}$.
\end{proposition}

In light of \eqref{length-pal}, we conclude that the episturmian conjugacy class of $\psi_u$ contains $|\pal(u)|+1$ morphisms. The next proposition provides a direct explanation of this fact.
\begin{proposition}\label{p:rcp}
    Let $\sigma\in\std(A)$ with $\sigma = \psi_u\circ\pi$ for $u\in A^*$, $\pi\in\perm(A)$. A morphism $\rho$ is a right conjugate of $\sigma$ if and only if $\rho = w^{-1}\sigma w$ for some prefix $w$ of $\pal(u)$.
\end{proposition}

\begin{proof}
    It suffices to prove the result for $\sigma=\psi_u$. Let $w$ be a prefix of $\pal(u)$. By \cref{l:pal-gcp}, $w$ and $\sigma(a)w$ are common prefixes of $\sigma(aba)$, hence $w$ is a prefix of $\sigma(a)w$. In particular, $\sigma = w\rho w^{-1}$ for some morphism $\rho$.

    Conversely, let $\rho$ be a right conjugate of $\sigma$ with $w$ such that $\rho = w^{-1}\sigma w$. Then, for every $a\neq b\in A$,
    \begin{equation*}
        \sigma(ab)w = \sigma(a)w\rho(b) = w\rho(ab).
    \end{equation*}
    Moreover $|w|\leq |\pal(u)|$, otherwise the prefix of $w$ of length $|\pal(u)|+1$ would be a common prefix of $\sigma(ab)$ and $\sigma(ba)$, contradicting the maximality of $\pal(u) = \gcp(\sigma(ab),\sigma(ba))$ from \cref{l:pal-gcp}. Hence $w$ is a prefix of $u$, as required.
\end{proof}

We next give two corollaries of \cref{p:rcp}. In the first one, we express the conjugacy index in terms of both the $\gcp$ operator and its dual, the \emph{greatest common suffix} operator $\gcs$. 
\begin{corollary}\label{c:index}
   Let $\sigma$ be an episturmian morphism. For every $a\neq b\in A$, the conjugacy index of $\sigma$ satisfies
    \begin{align*}
        \ind(\sigma) &= \frac{\Vert\sigma\Vert - |A|}{|A|-1} - |\gcp(\sigma(ab),\sigma(ba))|\\
                     &=|\gcs(\sigma(ab),\sigma(ba))|.
        \end{align*}
\end{corollary}

\begin{proof}
    It suffices to treat the case where $\sigma = w^{-1}\psi_uw$ with $u\in A^*$ and $|w|=\ind(\sigma)$. Recall that by \cref{p:rcp}, $w$ is a prefix of $\pal(u)$. Fix two distinct letters $a\neq b\in A$. By \cref{l:pal-gcp}, $\pal(u)a$ is a prefix of $\psi_u(ab)$, and therefore $w^{-1}\pal(u)a$ is a prefix of
    \begin{equation*}
        w^{-1}\psi_u(ab)w = (w^{-1}\psi_uw)(ab) = \sigma(ab).
    \end{equation*}
    It follows that $w^{-1}\pal(u) = \gcp(\sigma(ab),\sigma(ba))$ and using \eqref{length-pal},
    \begin{equation*}
        |w^{-1}\pal(u)| = |\pal(u)|-|w| =  \frac{\Vert\sigma\Vert - |A|}{|A|-1} - \ind(\sigma),
    \end{equation*}
    which gives the first equality.

    For the second equality, observe that $\sigma(ab)$ is a suffix of $bw$ or vice-versa. Since $|\sigma(ab)|\geq |\pal(u)| \geq |w|$ we conclude that the latter alternative holds.
\end{proof}

\begin{corollary}\label{c:gcs-gcp}
    Let $u\in A^*$. For every conjugate $\sigma$ of $\psi_u$ and every $a\neq b\in A$,
    \begin{equation*}
        \pal(u) = \gcs(\sigma(ab),\sigma(ba))\gcp(\sigma(ab),\sigma(ba)).
    \end{equation*}
\end{corollary}

\begin{proof}
    Let $w$ be the prefix of $\pal(u)$ such that $\sigma = w^{-1}\psi_uw$ (cf.~\cref{p:rcp}). Observe that for every $x\in A^*$, either $\sigma(x)$ is a suffix of $w$ or vice-versa. But by \cref{c:index}, we have that
    \begin{equation*}
        |\gcs(\sigma(ab),\sigma(ba))| = |w|,
    \end{equation*}
    hence $\gcs(\sigma(ab),\sigma(ba))=w$. A dual argument shows that $\gcp(\sigma(ab),\sigma(ba)) = w^{-1}\pal(u)$.
\end{proof}

We introduce with the next lemma the notion of minimal letters, which is well-defined for all episturmian morphisms which are not permutations.

\begin{lemma}\label{l:minletter}
    Let $\sigma$ be an episturmian morphism over $A$ which is not a permutation. There exists a unique letter $a_{\min}\in A$ such that $|\sigma(a_{\min})|<|\sigma(a)|$ for every $a\in A$, $a\neq a_{\min}$. Moreover if $\sigma=\psi_u$ for $u\in A^+$, then $a_{\min}$ is the last letter of $u$.
\end{lemma} 

\begin{proof}
    Observe that the property  of the first claim is preversed by conjugacy. Since, by \cref{t:richomme}, every episturmian morphism has a (unique) standard conjugate, we may assume that $\sigma$ is standard. Up to precomposing by an appropriate permutation, we may in fact suppose that $\sigma=\psi_u$ where $u\in A^+$. Write $u=va$ where $v\in A^*$, $a\in A$. For $b\in A\setminus\{a\}$ we have
    \begin{equation*}
        \psi_u(b) = \psi_v(\psi_a(b)) = \psi_v(ab) = \psi_v(a)\psi_v(b) = \psi_v(\psi_a(a))\psi_v(b) = \psi_u(a)\psi_v(b).
    \end{equation*}
    It follows that $\psi_u(a)$ is a proper prefix of every other $\psi_u(b)$, hence the result holds for $a_{\min} = a$.
\end{proof}

We call $a_{\min}$ the \emph{minimal letter} of $\sigma$. All episturmian morphisms in the same (non-trivial) conjugacy class share the same minimal letter. The following proposition highlights some of its key properties.

\begin{proposition}
    \label{p:minletter-epi}
    Let $\sigma$ be an episturmian morphism which is not a permutation. Then its minimal letter $a_{\min}$ satisfies the following three properties:
    \begin{enumerate}
        \item $\sigma(A)\subseteq A^*\sigma(a_{\min})\iff\ind(\sigma)\geq |\sigma(a_{\min})|$.
            \label{i:minletter-epi-1}
        \item $\sigma(A)\subseteq \sigma(a_{\min})A^*\iff\ind(\sigma)\leq\frac{\Vert\sigma\Vert-|A|}{|A|-1}-|\sigma(a_{\min})|$.
            \label{i:minletter-epi-2}
        \item $2|\sigma(a_{\min})| \leq \frac{\Vert\sigma\Vert - 1}{|A|-1}$.
            \label{i:minletter-epi-3}
    \end{enumerate}
\end{proposition}

\begin{proof}
    \ref{i:minletter-epi-1}. Let $\sigma= w^{-1}\psi_u w$ with $u\in A^+$ and $w$ the prefix of length $\ind(\sigma)$ of $\pal(u)$. We assume that $\ind(\sigma)=|w|\geq|\sigma(a_{\min})|$, and take $b\neq a_{\min}$ in $A$. By \cref{c:index}, $w$ is the greatest common suffix of $\sigma(ba_{\min})$ and $\sigma(a_{\min}b)$, hence $\sigma(a_{\min})$ is a suffix of $w$ and in turn a suffix of $\sigma(a_{\min}b)$. Since $|\sigma(b)|>|\sigma(a_{\min})|$, we conclude that $\sigma(a_{\min})$ is a suffix of $\sigma(b)$. Conversely assume that $\ind(\sigma) = |w|<|\sigma(a_{\min})|$. Then $\sigma(a_{\min})$ cannot be a suffix $\sigma(a_{\min}b)$, because this would contradict the maximality of $w$ as a common suffix of $\sigma(ba_{\min})$ and $\sigma(a_{\min}b)$.

    \ref{i:minletter-epi-2}. The proof is dual.

    \ref{i:minletter-epi-3}. For this part of the statement, we may assume that $\sigma = \psi_u$ for some $u\in A^+$. First we claim that for every $v\in A^*$:
    \begin{equation}\label{eq:ineq-max}
        \max\{|\psi_v(a)| \mathrel{;} a\in A\} \leq |\pal(v)|+1.
    \end{equation}
    We proceed by induction on $|v|$. When $v=\emptyw$, this is trivial. Assume that the result holds for $v\in A^*$ and let $a\in A$. Using Justin's dual formula \eqref{J'} and the induction hypothesis, we find
    \begin{align*}
        |\pal(va)|+1 &= |\pal(v)|+1+|\bar{\psi}_v(a)| = |\pal(v)|+1+|\psi_v(a)| \\
                     &\geq |\psi_v(b)|+|\psi_v(a)| = |\psi_v(ab)| \geq |\psi_{va}(b)|.
    \end{align*}
    This proves the claim.

    We conclude the proof of \ref{i:minletter-epi-3} by reducing it to the inequality \eqref{eq:ineq-max}, where $v\in A^*$ is such  that $u = va$, $a\in A$. According to  \cref{l:minletter}, $a = a_{\min}$. Moreover, we have $\psi_{va}(a)=\psi_v(a)$. Combining \eqref{eq:ineq-max} and \eqref{length-pal}, we get
    \begin{equation*}
        \frac{\Vert\psi_{va}\Vert - 1}{|A|-1} = |\psi_v(a)| + \frac{\Vert\psi_v\Vert-1}{|A|-1} = |\psi_v(a)| + |\pal(v)|+1 \geq 2|\psi_v(a)|.\qedhere
    \end{equation*}
\end{proof}

\begin{corollary}\label{c:min}
    Let $\sigma$ be an episturmian morphism over $A$ which is not a permutation. Then its minimal letter $a_{\min}$ is such that either  $\sigma(A)\subseteq A^*\sigma(a_{\min})$, or $\sigma(A)\subseteq \sigma(a_{\min})A^*$.
\end{corollary}

\begin{proof}
    If \ref{i:minletter-epi-1} from \cref{p:minletter-epi} does not hold, then by \ref{i:minletter-epi-3} the conjugacy index of $\sigma$ satisfies
    \begin{equation*}
        \ind(\sigma) \leq |\sigma(a_{\min})|-1 \leq \frac{\Vert\sigma\Vert - 1}{|A|-1} - |\sigma(a_{\min})| - 1 = \frac{\Vert\sigma\Vert - |A|}{|A|-1} - |\sigma(a_{\min})|,
    \end{equation*}
    which means that \ref{i:minletter-epi-2} holds.
\end{proof}

\begin{remark}
    For a standard episturmian morphism, the conjugacy index is $0$ and the greatest common prefix $\gcp(\sigma(ab) \mid a\neq b \in A)$ is a word  corresponding to  an extreme case of the  Fine and Wilf theorem, such as considered in \cite{Castelli1999} for three-letter alphabets and in \cite{Justin2000a} for the general case. These are palindromic prefixes of standard episturmian words. By \cref{c:index}, for a three-letter alphabet, the word $\gcp(\sigma(ab) \mid a\neq b \in A)$ has length $1/2(\Vert\sigma\Vert-3)$ (cf.\ \cite{Castelli1999}, Definition 5.2). 
 \end{remark}

\begin{example}
    Fix an alphabet $A = \{\ltr{a}_1,\cdots,\ltr{a}_d\}$ of size $d\geq 2$, and consider the $d$-bonacci morphism $\sigma=\psi_{\ltr{a}_1}\circ\pi$ from \cref{e:dbo}. Since $\Vert\sigma\Vert = 2|A|-1$, its conjugacy class consists of 2 elements, which are $\sigma$ and $\bar\sigma = \bar\psi_{\ltr{a}_1}\circ\pi$. In this case, the minimal letter $a_{\min} = \ltr{a}_d$ satisfies $|\sigma(\ltr{a}_d)| = 1$, which is exactly half the size of the conjugacy class. In other words, the conditions \ref{i:minletter-epi-1} and \ref{i:minletter-epi-2} from \cref{p:minletter-epi} do not overlap for this conjugacy class.
\end{example}

\begin{example}
    Let $\sigma = \psi_{\ltr{abca}}$. Then straightforward computations show that $|\sigma(a_{\min})| = |\sigma(\ltr{a})| = 7$ and 
    \begin{equation*}
        \frac{\Vert\sigma\Vert -1}{|A|-1} = \frac{31-1}{2} = 15.
    \end{equation*}
    Therefore the inequality in \ref{i:minletter-epi-3} from \cref{p:minletter-epi} is strict in this conjugacy class. Here are the conjugates of $\sigma$ with index 6, 7 and 8:
    \begin{gather*}
        \ltr{a}\mapsto \ltr{aabacab},\ \ltr{b}\mapsto \ltr{aabacababacab},\ \ltr{c}\mapsto\ltr{aabacabacab},\\
        \ltr{a}\mapsto \ltr{abacaba},\ \ltr{b}\mapsto \ltr{abacababacaba},\ \ltr{c}\mapsto\ltr{abacabacaba},\\
        \ltr{a}\mapsto \ltr{bacabaa},\ \ltr{b}\mapsto \ltr{bacababacabaa},\ \ltr{c}\mapsto\ltr{bacabacabaa}.
    \end{gather*}
    We see that the morphism with index 7 is simultaneously the first morphism in the conjugacy class satisfying \ref{i:minletter-epi-1}, and the last one satisfying \ref{i:minletter-epi-2}. In particular these two properties overlap.
\end{example}
 
\section{Rauzy graphs and return words}\label{s:Rauzy}

Rauzy graphs of  episturmian shift spaces are particularly convenient for their study. This is in fact how Arnoux and Rauzy first defined Arnoux--Rauzy words in their seminal paper \cite{Arnoux1991}, which initiated the study of episturmian words. 

In a shift $X$, we define the left and right Rauzy graphs $\Gamma_n$ and $\bar\Gamma_n$ as follows: 
\begin{itemize}
    \item in both cases, vertices are the elements of $\lang(X)\cap A^n$;
    \item each word $u\in\lang(X)\cap A^{n+1}$ gives an edge $u_{[0,n)}\to u_{(0,n]}$, labelled respectively by $u_{0}$ in $\Gamma_n$, and $u_{n}$ in $\bar\Gamma_n$.
\end{itemize}
        
In other words these two graphs have the same underlying shape, but are labeled differently. An example is given in \cref{f:tetrabonacci}.

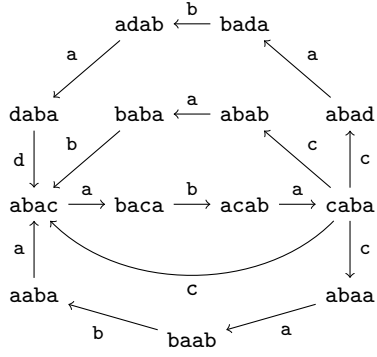
\begin{figure}\centering
\begin{tikzpicture}[xscale=1.4,yscale=1.2,>={to}]
    \node (abac) at (0,0)      {$\ltr{abac}$} ;
    \node (caba) at (3,0)      {$\ltr{caba}$} ;
    \node (baca) at (1,0)      {$\ltr{baca}$} ;
    \node (acab) at (2,0)      {$\ltr{acab}$} ;
    \node (abab) at (2,1)      {$\ltr{abab}$} ;
    \node (baba) at (1,1)      {$\ltr{baba}$} ;
    \node (abaa) at (3,-1)     {$\ltr{abaa}$} ;
    \node (baab) at (1.5,-1.5) {$\ltr{baab}$} ;
    \node (aaba) at (0,-1)     {$\ltr{aaba}$} ;
    \node (abad) at (3,1)      {$\ltr{abad}$} ;
    \node (bada) at (2,2)      {$\ltr{bada}$} ;
    \node (adab) at (1,2)      {$\ltr{adab}$} ;
    \node (daba) at (0,1)      {$\ltr{daba}$} ;
    \draw[->] (abac) edge node [font=\footnotesize,auto] {$\ltr{a}$} (baca);
    \draw[->] (baca) edge node [font=\footnotesize,auto] {$\ltr{b}$} (acab);
    \draw[->] (acab) edge node [font=\footnotesize,auto] {$\ltr{a}$} (caba);
    \draw[->] (caba) edge node [font=\footnotesize,auto] {$\ltr{c}$} (abaa);
    \draw[->] (abaa) edge node [font=\footnotesize,auto] {$\ltr{a}$} (baab);
    \draw[->] (baab) edge node [font=\footnotesize,auto] {$\ltr{b}$} (aaba);
    \draw[->] (aaba) edge node [font=\footnotesize,auto] {$\ltr{a}$} (abac);
    \draw[->] (caba) edge node [font=\footnotesize,auto,swap] {$\ltr{c}$} (abab);
    \draw[->] (abab) edge node [font=\footnotesize,auto,swap] {$\ltr{a}$} (baba);
    \draw[->] (baba) edge node [font=\footnotesize,auto,swap] {$\ltr{b}$} (abac);
    \draw[->] (caba) edge node [font=\footnotesize,auto,swap] {$\ltr{c}$} (abad);
    \draw[->] (abad) edge node [font=\footnotesize,auto,swap] {$\ltr{a}$} (bada);
    \draw[->] (bada) edge node [font=\footnotesize,auto,swap] {$\ltr{b}$} (adab);
    \draw[->] (adab) edge node [font=\footnotesize,auto,swap] {$\ltr{a}$} (daba);
    \draw[->] (daba) edge node [font=\footnotesize,auto,swap] {$\ltr{d}$} (abac);
    \draw[->,bend left=50] (caba) edge node [auto] {$\ltr{c}$} (abac) ;
\end{tikzpicture}
    \caption{Rauzy graph $\Gamma_4$ in the \enquote{tetrabonacci} shift ($d$-bonacci with $d=4$).}
    \label{f:tetrabonacci}
\end{figure} 

In episturmian shift spaces, the shape of the Rauzy graphs are easily understood: excluding the degenerate case where $\lesp_n$ does not exist, they consist of a single (possibly empty) path $\lesp_n\to\risp_n$, called the \emph{inner branch}, together with disjoint paths $\risp_n\to \lesp_n$, called the \emph{outer branches}. Note that when $\lesp_n=\risp_n$, the inner branch is reduced to a single vertex. In a strict episturmian shift, all of the Rauzy graphs have $|A|$ outer branches. See \cref{f:rauzy-shapes} for an illustration of possible shapes of Rauzy graphs in episturmian shifts.

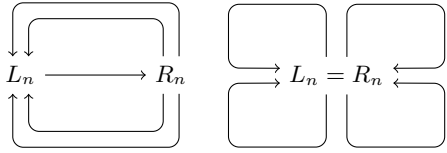
\begin{figure}
    \begin{tabular}{cc}
        \begin{tikzpicture}[>=to]
            \node (L) at (0,0) {$L_n$} ;
            \node (R) at (2,0) {$R_n$} ; 
    
            \draw[->] (L) to (R) ;
            \draw[->,rounded corners] ([xshift=3pt]R.90) to [yshift=6pt] ++(0,.5) to [yshift=6pt,xshift=-6pt] ++(-2,0) to ([xshift=-3pt]L.90) ; 
            \draw[->,rounded corners] ([xshift=-3pt]R.90) to ++(0,.5) to [xshift=6pt] ++(-2,0) to ([xshift=3pt]L.90) ; 
            \draw[->,rounded corners] ([xshift=3pt]R.270) to [yshift=-6pt] ++(0,-.5) to [yshift=-6pt,xshift=-6pt] ++(-2,0) to ([xshift=-3pt]L.270) ; 
            \draw[->,rounded corners] ([xshift=-3pt]R.270) to ++(0,-.5) to [xshift=6pt] ++(-2,0) to ([xshift=3pt]L.270) ; 
        \end{tikzpicture}
        &
        \begin{tikzpicture}[xscale=1.3,yscale=1.4,>={to}]
            \node (L) at (0,0) {$L_n = R_n$} ;
    
            \draw[->,rounded corners] ([xshift=3pt]L.90) to ++(0,.5) to ++(1,0) to [yshift=-.7ex] ++(0,-.5) to ([yshift=2pt]L.0) ; 
            \draw[->,rounded corners] ([xshift=-3pt]L.90) to ++(0,.5) to ++(-1,0) to [yshift=-.7ex] ++(0,-.5) to ([yshift=2pt]L.180) ; 
            \draw[->,rounded corners] ([xshift=3pt]L.270) to ++(0,-.5) to ++(1,0) to [yshift=.7ex] ++(0,.5) to ([yshift=-2pt]L.0) ; 
            \draw[->,rounded corners] ([xshift=-3pt]L.270) to ++(0,-.5) to ++(-1,0) to [yshift=.7ex] ++(0,.5) to ([yshift=-2pt]L.180) ; 
        \end{tikzpicture}
    \end{tabular}
    \caption{Possible shapes of Rauzy graphs in a strict episturmian shift on 4 letters.}
    \label{f:rauzy-shapes}
\end{figure}

\begin{definition}
    We say that $w\in\lang(X)$ is an \emph{inner word} if it lies on the inner branch of $\Gamma_n$ ($n=|w|$), including if $w = \lesp_n$ or $\risp_n$. Otherwise, we say that $w$ is an \emph{outer word}.
\end{definition}

We will denote by $U_n$ the label of the inner branch in $\Gamma_n$ and by $\bar U_n$ its label in $\bar\Gamma_n$. In $\Gamma_n$, the outer branches all end with different letters; for every left extension $a$ of $\lesp_n$, let $V_{a,n}$ be the label of the outer branch ending with $a$. Dually, let $\bar V_{a,n}$ be the label of the unique outer branch in $\bar\Gamma_n$ starting with $a$. The next theorem, due to Arnoux and Rauzy (for three-letter alphabets), describes the evolution of the Rauzy graphs in terms of these words; see also \cite{Cassaigne2006}.

\begin{theorem}[\cite{Arnoux1991}]\label{t:arnoux-rauzy}
    Let $X$ be a non-degenerate episturmian shift space. Fix $n\in\nn$ and let $a\in A$ be such that $\lesp_na = \lesp_{n+1}$ and $a\risp_n = \risp_{n+1}$. Then, the following relations hold between the labels of the inner and outer branches.
    \begin{itemize}
        \item If $\lesp_n \neq \risp_n$ then:
        \begin{enumerate}[series=fente-eclatement]
            \item $U_{n+1}a = U_n$ and $V_{b,n+1} = aV_{b,n}$ for all $b\in A$;\label{i:fente}
            \item $a\bar U_{n+1} = \bar U_n$ and $\bar V_{b,n+1} = \bar V_{b,n}a$ for all $b\in A$.
        \end{enumerate}
        \item If $\lesp_n = \risp_n$ then:
        \begin{enumerate}[resume=fente-eclatement]
            \item $U_{n+1}a = V_{a,n}$, $V_{a,n+1} = a$, and $V_{b,n+1} = aV_{b,n}$ for $b\neq a\in A$;\label{i:eclatement}
            \item $a\bar U_{n+1} = \bar V_{a,n}$, $\bar V_{a,n+1} = a$, and $\bar V_{b,n+1} = \bar V_{b,n}a$ for $b\neq a\in A$.
        \end{enumerate}
    \end{itemize}
\end{theorem}

Using this as a starting point, we give a description for sets of return words in episturmian shift spaces. The description, stated in \cref{t:return} below, is phrased in terms of \emph{directive words}, which we recall next. 

First, let us extend the notion of palindromic closure to infinite words. Let $\bs d\in A^\nn$. By Justin's second formula \eqref{J'}, $\pal(\bs d_{[0,n)})$ is a (strict) prefix of $\pal(\bs d_{[0,n+1)})$ for every $n\in\nn$. It follows that there exists a unique infinite word, denoted $\pal(\bs d)$, having $\pal(\bs d_{[0,n)})$ as prefix for every $n\in\nn$. This naturally extends $\pal$ to a map $A^\nn\to A^\nn$. 

By \cite{Droubay2001}, Theorem~1, this extension of $\pal$ gives a bijection from $A^\nn$ onto the set of standard episturmian words. In particular $\lesp_n$, whenever it exists, is a prefix of $\pal(\bs d)$ (\cite{Droubay2001}, Proposition~5). Moreover, the words $\pal(\bs d_{[0,k)})$ are exactly the prefixes of $\pal(\bs d)$ which are bispecial whenever they are right special. All of these words are bispecial, or equivalently $\pal(\bs d)$ is non-degenerate, precisely when $\bs d$ has at least two letters occurring infinitely often (\cite{Droubay2001}, Theorem 3).

Since an episturmian shift space $X$ contains \emph{exactly one} standard episturmian word, we conclude that $X$ is the orbit-closure of $\pal(\bs d)$ for a unique $\bs d\in A^\nn$. We call $\bs d$ the \emph{directive word} of the shift $X$. Note that the shift $X$ does not need here to be generated by an episturmian morphism. See \cite{Glen2009a} for more details on directive words of episturmian shift spaces. 

\begin{lemma}[\cite{Justin2000}, Lemma~3.2]
    Let $X$ be an episturmian shift space with directive word $\bs d\in A^\nn$. Take $w\in\lang(X)$ and let $n=n(w)$ be the least positive integer such that $w$ occurs in $\pal(\bs d_{[0,n)})$. Then $w$ occurs exactly once in $\pal(\bs d_{[0,n)})$.
\end{lemma}

Given a word $u$ and positive integer $0\leq i\leq |\pal(u)|$, we let $\psi_u^{(i)}$ be the $i$-th right conjugate of $\psi_u$; recall that $\psi_u$ has $|\pal(u)|+1$ conjugates by \cref{p:rcp}, so this is well-defined. Dually, we let $\bar\psi_u^{(i)}$ be the $i$-th \emph{left} conjugate of $\bar\psi_u$. Observe that $\psi_u^{(i)} = \bar\psi_u^{(m-i)}$ where $m = |\pal(u)|$.

The next theorem combines two results from \cite{Justin2000}. In particular, the first part of the proof is essentially \cite[Theorem 4.4]{Justin2000}; we provide an argument for the sake of completeness. We also recall in passing that episturmian shifts are always minimal (\cite{Droubay2001}, Theorem 2).

\begin{theorem}\label{t:return}
    Let $X$ be a non-degenerate episturmian shift space with directive word $\bs d\in A^\nn$. Take $w\in\lang(X)$ and let $n=n(w)$ be the least positive integer such that $w$ occurs in $\pal(\bs d_{[0,n)})$. Assume that $w$ occurs at index $\ell$ in $\pal(\bs d_{[0,n)})$ and let $\ell' = |\pal(\bs d_{[0,n)})| - \ell - |w|$. Then the following equalities hold:
    \begin{equation*}
        \ret_X(w) = \psi_{\bs d_{[0,n)}}^{(\ell)}(A),\qquad \bar\ret_X(w) = \bar\psi_{\bs d_{[0,n)}}^{(\ell')}(A).
    \end{equation*}
\end{theorem}

\begin{proof}
    Let $p=|w|$. First, assume that $w$ is bispecial, which is the case precisely when $w = \lesp_{p} = \risp_{p} = \pal(\bs d_{[0,n)})$. In this case we have $\ret_X(w) = \{V_{a,p}\mid a\in A\}$ and $\bar\ret_X(w) = \{\bar V_{a,p}\mid a\in A\}$. Let us show by induction on $n$ that $V_{a,p} = \psi_{d_{[0,n)}}(a)$ for every $a\in A$; the dual of this argument gives $\bar V_{a,p} = \bar\psi_{d_{[0,n)}}(a)$. If $n=0$, then $w = \emptyw$ and $\psi_{\emptyw}(a) = a = V_{a,0}$, as expected. 

    For the induction step assume that $n>0$, let $v = \pal(\bs d_{[0,n-1)})$ and $q=|v|$. Using Justin's second formula \eqref{J'}, we find that
    \begin{equation*}
       w = \pal(\bs d_{[0,n)}) = \pal(\bs d_{[0,n-1)})\bar\psi_{\bs d_{[0,n-1)}}(\bs d_n) = v\bar\psi_{\bs d_{[0,n-1)}}(\bs d_n).
    \end{equation*}
    But notice that $\bar\psi_x(a)$ starts with $a$, for every $a\in A$ and $x\in A^*$; in particular it follows that $v\bs d_n = \lesp_q\bs d_n = \lesp_{q+1}$. Since $v=L_q=R_q$, we have by Theorem~\ref{t:arnoux-rauzy}~\ref{i:eclatement}:
    \begin{gather*}
        U_{q+1}\bs d_n = V_{\bs d_n,q},\qquad V_{\bs d_n,q+1} = \bs d_n,\\
        V_{b,q+1} = \bs d_nV_{b,q} \quad \text{for} \ b\neq \bs d_n.
    \end{gather*}
    Since $\lesp_i\neq\risp_i$ for $q<i<p$, we can then apply \cref{t:arnoux-rauzy}~\ref{i:eclatement} $p-q$ times to conclude that 
    \begin{equation*}
        V_{b,p} = U_{q+1}V_{\bs d_n,q+1} = 
        \begin{cases}
                V_{\bs d_n,q} & \text{ if }b = \bs d_n\\
                V_{\bs d_n,q}V_{b,q} & \text{ if }b \neq \bs d_n.
        \end{cases}
    \end{equation*}
    Finally, using the induction hypothesis, we have:
    \begin{gather*}
        V_{\bs d_n,p} = V_{\bs d_n,q} = \psi_{\bs d_{[0,n-1)}}(\bs d_n) = \psi_{\bs d_{[0,n)}}(\bs d_n),\\
        V_{b,p} = V_{\bs d_n,q}V_{b,q} = \psi_{\bs d_{[0,n-1)}}(\bs d_nb) = \psi_{\bs d_{[0,n)}}(b) \quad \text{for} \ b\neq d_n.
    \end{gather*}
    
    It remains to handle the case where $w$ is not bispecial. Let $\pal(\bs d_{[0,n)]}) = fwg$ where $|f|=\ell$ and $|g|=\ell'$. Since $f$ is a prefix of $\pal(\bs d_{[0,n)]})$, it follows that $f^{-1}\psi_{\bs d_{[0,n)}}f = \psi_{\bs d_{[0,n)}}^{(\ell)}$. Next since $fwg$ is bispecial, we have by the first part of the proof that
    \begin{equation*}
        \ret_X(fwg) = \psi_{\bs d_{[0,n)}}(A).
    \end{equation*}
    Finally, by \cite{Justin2000}, Corollary 4.1, we get
    \begin{equation*}
        \ret_X(w) = f^{-1}\ret_X(fwg)f = f^{-1}\psi_{\bs d_{[0,n)}}(A)f = \psi_{\bs d_{[0,n)}}^{(\ell)}(A).
    \end{equation*}

    The second formula is obtained in a similar way: observe that $\bar\ret_X(fwg) = \bar\psi_{\bs d_{[0,n)}}(A)$ and $g\bar\psi_{\bs d_{[0,n)}}g^{-1} = \psi_{\bs d_{[0,n)}}^{(\ell')}$, and then apply the dual form of \cite{Justin2000}, Corollary~4.1.
\end{proof}

\begin{remark}
    Since episturmian morphisms are invertible as morphisms of the free group, we recover from the previous proposition the fact that episturmian shift spaces obey the so-called \emph{Return Theorem}, which states that in a dendric shift space (the common generalization of episturmian shifts and codings of interval exchanges discussed in the introduction), every return set is a basis for the free group over the alphabet of the shift~\cite{Berthe2015}.
\end{remark}

The next corollary reinterprets the above theorem in terms of $\gcs$ and $\gcp$.
\begin{corollary}
    Let $X$ be a non-degenerate episturmian shift and $u\in\lang(X)$. 
    \begin{enumerate}
        \item For every pair of distinct return words $r\neq s\in\ret_X(u)$, the word $w = \gcs(rs,sr)\gcp(rs,sr)$ is the shortest bispecial factor in which $u$ occurs. Moreover, $u$ occurs in $w$ at index $|\gcs(rs,sr)|$ and nowhere else.
        \item For every pair of distinct return words $r\neq s\in\bar\ret_X(u)$, the word $w = \gcs(rs,sr)\gcp(rs,sr)$ is the shortest bispecial factor in which $u$ occurs. Moreover, $u$ occurs in $w$ at index $|w|-|\gcp(rs,sr)|-|u|$ and nowhere else.
    \end{enumerate}
\end{corollary}

In particular, observe that the word $\gcs(rs,sr)\gcp(rs,sr)$ in the above corollary is independent of the choice of $r\neq s$ in $\ret_X(u)$ or $\bar\ret_X(u)$, and is equal to $\pal(\dee)$ for a unique prefix $\dee$ of the directive word of $X$. Thus, the following notation is well-defined.
\begin{notation}\label{n:return}
 Let $X$ be a non-degenerate episturmian shift space, $u\in\lang(X)$, and $r\neq s\in\ret_X(u)$. Then we use the following notation:
 \begin{equation*}
     \dee(u) = \pal^{-1}(\gcs(rs,sr)\gcp(rs,sr)),\quad \ell(u) = |\gcs(rs,sr)|,\quad \ell'(u) = |\gcp(rs,sr)|.
 \end{equation*}
All of these are independent of $r$ and $s$ and satisfy 
\begin{equation*}
    \ret_X(u) = \psi_{\dee(u)}^{(\ell(u))}(A),\qquad \bar\ret_X(u) = \bar\psi_{\dee(u)}^{(\ell'(u))}(A).
\end{equation*}
This is summarized in the commutative diagrams below.
\begin{equation*}
    \begin{tikzcd}
        \lang(X) \rar{\dee\times \ell} \drar[swap]{\ret_X} & A^*\times\nn \dar{\psi} \\
        & 2^{A^*}
    \end{tikzcd}
    \quad
    \begin{tikzcd}
        \lang(X) \rar{\dee\times \ell'} \drar[swap]{\bar\ret_X} & A^*\times\nn \dar{\bar\psi} \\
        & 2^{A^*}
    \end{tikzcd}
\end{equation*}
\end{notation}

For inner words, the pair $(\dee,\ell,\ell')$ admits the following simple description. We recall that $U_n$ is  the label of the inner branch in $\Gamma_n$.

\begin{lemma}\label{l:dl}
    Let $X$ be an episturmian shift directed by $\bs d$. Let $u$ be its $i$-th inner word of length $p$ (starting from $u=\lesp_p$ when $i=0$). Let $n$ be the number of bispecial prefixes of $L_p$. Then $\dee(w) = \bs d_{[0,n)}$, $\ell(w)=i$ and $\ell'(w) = |U_p|-i$.
\end{lemma}

\begin{proof}
    First observe that $\lesp_pU_p = \bar U_p\risp_p$ is the smallest bispecial factor which contains $\lesp_p$. Moreover, the $i$-th factor of length $p$ in $\lesp_pU_p$ is precisely the $i$-th word on the inner branch in the Rauzy graph $\Gamma_p$. We claim that $\lesp_pU_p$ is also the smallest bispecial factor containing each of the inner words of length $p$. Indeed, any shorter bispecial word containing a word on the inner branch would have $\lesp_p$ as a prefix of length $p$, contradicting our first observation. Since $\lesp_pU_p$ is the $(n+1)$-th bispecial word in $\lang(X)$, it is equal to $\pal(\bs d_{[0,n)})$ and the result follows.
\end{proof}

\begin{example}\label{e:dl}
    Consider the ternary morphism $\sigma = \psi_{\ltr{abb}}\circ(\ltr{a}\:\ltr{c}\:\ltr{b})$. It is a primitive episturmian morphism, thus its shift space $X = X_\sigma$ is a strict episturmian shift. The directive word of $X$ is the periodic infinite word
    \begin{equation*}
        \bs d = (\ltr{abb\,caa\,bcc})^\infty.
    \end{equation*}
    The Rauzy graph $\Gamma_8$ is depicted in \cref{f:rauzy-dl} with each word accompanied by the values of $\dee$ and $\ell$.
    Take for example the word $w=\ltr{acababab}$, which is the first word on the outer branch of $\Gamma_8$ ending with the letter $\ltr{b}$. This word first appears in $\pal(\bs d_{[0,7)})$ at index 26; thus $\dee(w) = \ltr{abbcaab}$ and $\ell(w) = 26$. The corresponding return set is given by $\ret_\sigma(w) = \psi_{\ltr{abbcaab}}^{(26)}(\{\ltr{a},\ltr{b},\ltr{c}\})$, which after computation amounts to:
    \begin{align*}
        \{ & \ltr{acabababacababaababacababaabab},\\
        &\ltr{acabababacababaababacababaababacababaabab},\\
        &\ltr{acabababacababaababacababaababacababacababaababacababaabab}\}.
    \end{align*}
    
    \begin{figure}
    \begin{tikzpicture}[font=\small, xscale=2.5]
        \node[label = {[font=\scriptsize]above:{$(\ltr{abbca},6)$}}] (ababaaba) at (3,3)   {$\ltr{ababaaba}$} ;
        \node[label = {[font=\scriptsize]below:{$(\ltr{abbcaab},27)$}}] (cabababa) at (2,-2)  {$\ltr{cabababa}$} ;
        \node[label = {[font=\scriptsize]left :{$(\ltr{abbca},9)$}}] (baababac) at (0,2)   {$\ltr{baababac}$} ;
        \node[label = {[font=\scriptsize]right:{$(\ltr{abbca},4)$}}] (acababaa) at (3,1)   {$\ltr{acababaa}$} ;
        \node[label = {[font=\scriptsize]left :{$(\ltr{abbc},0)$}}] (ababacab) at (0,0)   {$\ltr{ababacab}$} ;
        \node[label = {[font=\scriptsize]above:{$(\ltr{abbc},2)$}}] (abacabab) at (2,0)   {$\ltr{abacabab}$} ;
        \node[label = {[font=\scriptsize]right:{$(\ltr{abbcaab},26)$}}] (acababab) at (3,-1)  {$\ltr{acababab}$} ;
        \node[label = {[font=\scriptsize]below:{$(\ltr{abbcaab},28)$}}] (abababac) at (1,-2)  {$\ltr{abababac}$} ;
        \node[label = {[font=\scriptsize]above:{$(\ltr{abbcaabc},56)$}}] (acababac) at (2,1.5)   {$\ltr{acababac}$} ;
        \node[label = {[font=\scriptsize]left :{$(\ltr{abbcaab},29)$}}] (bababaca) at (0,-1)  {$\ltr{bababaca}$} ;
        \node[label = {[font=\scriptsize]above:{$(\ltr{abbca},7)$}}] (babaabab) at (1.5,3) {$\ltr{babaabab}$} ;
        \node[label = {[font=\scriptsize]above:{$(\ltr{abbcaabc},57)$}}] (cababaca) at (1,1.5)   {$\ltr{cababaca}$} ;
        \node[label = {[font=\scriptsize]above:{$(\ltr{abbc},1)$}}] (babacaba) at (1,0)   {$\ltr{babacaba}$} ;
        \node[label = {[font=\scriptsize]left :{$(\ltr{abbca},10)$}}] (aababaca) at (0,1)   {$\ltr{aababaca}$} ;
        \node[label = {[font=\scriptsize]right:{$(\ltr{abbca},5)$}}] (cababaab) at (3,2)   {$\ltr{cababaab}$} ;
        \node[label = {[font=\scriptsize]above:{$(\ltr{abbca},8)$}}] (abaababa) at (0,3)   {$\ltr{abaababa}$} ;
        \node[label = {[font=\scriptsize]right:{$(\ltr{abbc},3)$}}] (bacababa) at (3,0)   {$\ltr{bacababa}$} ;
        \draw[->] (ababacab) to node [auto,font=\footnotesize] {$\ltr{a}$} (babacaba);
        \draw[->] (babacaba) to node [auto,font=\footnotesize] {$\ltr{b}$} (abacabab);
        \draw[->] (abacabab) to node [auto,font=\footnotesize] {$\ltr{a}$} (bacababa);
        
        \draw[->] (bacababa) to node [auto,swap,font=\footnotesize] {$\ltr{b}$} (acababaa);
        \draw[->] (acababaa) to node [auto,swap,font=\footnotesize] {$\ltr{a}$} (cababaab);
        \draw[->] (cababaab) to node [auto,swap,font=\footnotesize] {$\ltr{c}$} (ababaaba);
        \draw[->] (ababaaba) to node [auto,swap,font=\footnotesize] {$\ltr{a}$} (babaabab);
        \draw[->] (babaabab) to node [auto,swap,font=\footnotesize] {$\ltr{b}$} (abaababa);
        \draw[->] (abaababa) to node [auto,swap,font=\footnotesize] {$\ltr{a}$} (baababac);
        \draw[->] (baababac) to node [auto,swap,font=\footnotesize] {$\ltr{b}$} (aababaca);
        \draw[->] (aababaca) to node [auto,swap,font=\footnotesize] {$\ltr{a}$} (ababacab);
        
        \draw[->] (bacababa) to node [auto,swap,font=\footnotesize] {$\ltr{b}$} (acababac);
        \draw[->] (acababac) to node [auto,swap,font=\footnotesize] {$\ltr{a}$} (cababaca);
        \draw[->] (cababaca) to node [auto,swap,font=\footnotesize] {$\ltr{c}$} (ababacab);
        
        \draw[->] (bacababa) to node [auto,font=\footnotesize] {$\ltr{b}$} (acababab);
        \draw[->] (acababab) to node [auto,font=\footnotesize] {$\ltr{a}$} (cabababa);
        \draw[->] (cabababa) to node [auto,font=\footnotesize] {$\ltr{c}$} (abababac);
        \draw[->] (abababac) to node [auto,font=\footnotesize] {$\ltr{a}$} (bababaca);
        \draw[->] (bababaca) to node [auto,font=\footnotesize] {$\ltr{b}$} (ababacab);
    \end{tikzpicture}
    \caption{Rauzy graph $\Gamma_8$ from \cref{e:dl} with pairs $(\dee,\ell)$.}
    \label{f:rauzy-dl}
    \end{figure}
\end{example}

\section{Proof of the main result}\label{s:proofmain}

In this section, we establish some lemmas leading to the proof of \cref{t:main}.  It relies on a careful study of the inner branch of the Rauzy graphs $\Gamma_k$  of order $k=|\sigma(a_{\min}L_n)|$, where $n\in\nn$ is such that $\lesp_n=\risp_n$, where one   is able to locate  either
$\sigma(a_{\min} L_n)$, or  $\sigma(R_na_{\min})$.

Consider a primitive standard episturmian morphism $\sigma$. Let $a_{\min}$ be the minimal letter of $\sigma$, $j = |\sigma(a_{\min})|$ and $m = \frac{\Vert\sigma\Vert-|A|}{|A|-1}$. Recall from \eqref{length-pal} and \cref{p:rcp} that the cardinality of the conjugacy class of $\sigma$ is equal to $m+1$. Thus, if we denote by $\sigma^{(i)}$ the $i$-th right conjugate of $\sigma$, then $\sigma^{(m)}$ is its last conjugate. Moreover by \cref{p:minletter-epi}, we have $2j\leq m+1$, hence the morphisms 
\begin{equation*}
    \sigma^{(j)},\ \sigma^{(j+1)},\ \ldots,\ \sigma^{(m)}
\end{equation*}
cover at least half of the conjugacy class. For similar reasons the following morphisms also cover at least half of the conjugacy class:
\begin{equation*}
    \sigma^{(0)},\ \sigma^{(1)},\ \ldots,\ \sigma^{(m-j)}.
\end{equation*}

In the proof of the main result, we consider for a given $n\in\nn$ the four words:
\begin{equation*}
    a_{\min}\lesp_n,\quad a_{\min}\risp_n,\quad \risp_na_{\min},\quad\lesp_na_{\min}.
\end{equation*}
The first and third are always in $\lang(\sigma)$, though the second and fourth  need not be.  When $L_n=R_n$, the four words are  in $\lang(\sigma)$. As a first step towards the proof of the main result, we establish the following lemma which locates the images of these four words in the appropriate Rauzy graph according to $\ind(\sigma).$
\begin{lemma}\label{l:proof-1}
    Let $\sigma$ be a primitive episturmian morphism over $A$ and let $a_{\min}$ be its minimal letter. Let $j = |\sigma(a_{\min})|$ and $m = \frac{\Vert\sigma\Vert-|A|}{|A|-1}$. For every $n\in\nn$:
    \begin{enumerate}
        \item $\sigma(a_{\min}L_n)$ is left special when $\ind(\sigma) = j$;
            \label{i:left-special}
        \item $\sigma(a_{\min}R_n)$ is right special when $\ind(\sigma) = m$;
            \label{i:right-special}
        \item $\sigma(R_na_{\min})$ is right special when $\ind(\sigma)=m-j$;
            \label{i:right-special-dual}
        \item $\sigma(L_na_{\min})$ is left special when $\ind(\sigma)=0$;
            \label{i:left-special-dual}
    \end{enumerate}
\end{lemma}

\begin{proof}
    \ref{i:left-special}. Assume that $\ind(\sigma)=j$ and let $c$ be the last letter of $\sigma(a_{\min})$. By \cref{p:minletter-epi}, $\sigma(a_{\min})$ is a suffix of $\sigma(a)$ for every letter $a\in A$, so the word $\sigma(a)\sigma(a_{\min})^{-1}$ is well-defined. Let $b_a$ be the last letter of $\sigma(a)\sigma(a_{\min})^{-1}$. Consider the morphism $\rho = c\sigma c^{-1}$, which is the conjugate of $\sigma$ of index $j-1$. Applying \cref{p:minletter-epi} again, we have $\rho(A) \not\subseteq A^*\rho(a_{\min})$, hence there is a letter $a\in A$ such that $b_a\neq c$. Since $\sigma$ is strict (\cref{p:strict}) this letter $a$ must be a left extension of $\lesp_n$. It follows that $\sigma(a_{\min}L_n)$ has at least two left extensions, namely $b_a$ and $c$, hence it is left special. 

    \ref{i:right-special}. We now assume that $\ind(\sigma) = m$. Note that $\sigma(a_{\min}\risp_n)$ belongs to $\lang(\sigma)$ (even though $a_{\min}\risp_n$ might not). Indeed, consider the left extension $a$ of $\risp_n$ such that $a\risp_n = \risp_{n+1}$. Then, as $m\geq j$, \cref{p:minletter-epi} implies that $\sigma(a_{\min}\risp_n)$ is a suffix of $\sigma(a\risp_n)$. Finally it suffices to notice that, since $\sigma$ is the last member of its conjugacy class, the image of two distinct letters have different first letters. Hence, taking the first letters of images of distinct right extensions of $a\risp_n$ yields distinct right extensions of $\sigma(a_{\min}\risp_n)$.

    The proofs of \ref{i:right-special-dual} and \ref{i:left-special-dual} are dual.
\end{proof}

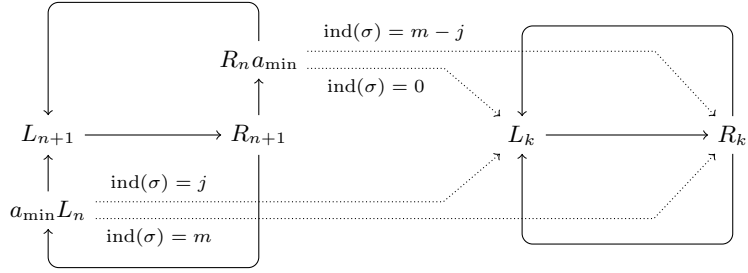
\begin{figure}
    \begin{tikzpicture}[xscale=1.4,>=to]
        \node (Ln+1) at (0,0) {$\lesp_{n+1}$} ;
        \node (Rn+1) at (2,0) {$\risp_{n+1}$} ;
        \node (Ln) at (0,-1) {$a_{\min}\lesp_n$} ;
        \node (Rn) at (2,1) {$\risp_na_{\min}$} ;
        \node(Lk) at (4.5,0) {$\lesp_k$} ;
        \node(Rk) at (6.5,0) {$\risp_k$} ;
        \draw[->] (Ln+1) to (Rn+1) ;
        \draw[->] (Ln) to (Ln+1) ;
        \draw[->,rounded corners] (Rn+1.south) to ++(0,-1.5) to ++(-2,0) to (Ln);
        \draw[->] (Rn+1) to (Rn) ;
        \draw[<-,rounded corners] (Ln+1.north) to ++(0,1.5) to ++(2,0) to (Rn);
        \draw[->] (Lk) to (Rk) ;
        \draw[->,rounded corners] (Rk.north) to ++(0,1.2) to ++(-2,0) to (Lk) ;
        \draw[->,rounded corners] (Rk.south) to ++(0,-1.2) to ++(-2,0) to (Lk) ;
        \draw[->,densely dotted] (Rn.10) to  node [near start,above,font=\scriptsize] {$\ind(\sigma)=m-j$} ++(3.3,0) to (Rk) ;
        \draw[->,densely dotted] (Rn.350) to node [pos=.5,below,font=\scriptsize] {$\ind(\sigma)=0$} ++(1.3,0) to (Lk) ;
        \draw[->,densely dotted] (Ln.350) to node [pos=.2,below,font=\scriptsize] {$\ind(\sigma)=m$} ++(3,0) to ++(2.3,0) to (Rk) ;
        \draw[->,densely dotted] (Ln.10) to node [pos=.2,above,font=\scriptsize] {$\ind(\sigma)=j$} ++(3,0) to ++(.3,0) to (Lk) ;
    \end{tikzpicture}
    \caption{The four cases in \cref{l:proof-1} when $L_n=R_n$  with $k=|\sigma(a_{\min} L_n)|$.}
    \label{f:proof-1}
\end{figure}

\cref{f:proof-1} depicts the different parts of the above lemma in the case where $\lesp_n=\risp_n$, which is of particular interest for the proof of the main result. 

\begin{lemma}\label{l:proof-2}
    Let $\sigma$ be an episturmian morphism with minimal letter $a_{\min}$. Let $j = |\sigma(a_{\min})|$, $m = \frac{\Vert\sigma\Vert-|A|}{|A|-1}$ and $k = |\sigma(a_{\min}\lesp_n)|$. For every sufficiently large $n\in\nn$ such that $\lesp_n = \risp_n$, the following holds:
    \begin{enumerate}
        \item $\sigma(a_{\min}\lesp_n)$ is the $(\ind(\sigma)-j)$-th inner word in $\Gamma_k$ provided $\ind(\sigma)\geq j$;
        \item $\sigma(\risp_na_{\min})$ is the $\ind(\sigma)$-th inner word in $\bar\Gamma_k$ provided $\ind(\sigma)\leq m-j$.
    \end{enumerate}
\end{lemma}

\begin{proof}
    Assume that $\sigma$ is standard and let $\sigma^{(i)}$ be its $i$-th conjugate, for $0\leq i\leq m$. Consider the sequence of words:
    \begin{equation}\label{eq:path}
        [\sigma^{(j)}(a_{\min}\lesp_n),\ \sigma^{(j+1)}(a_{\min}\lesp_n),\ \ldots,\ \sigma^{(m)}(a_{\min}\lesp_n)].
    \end{equation}
    The content of the lemma is equivalent to the claim that \eqref{eq:path} is exactly the inner branch of $\Gamma_k$. By \cref{l:proof-1}, $\sigma^{(j)}(a_{\min}\lesp_n) = L_k$ and $\sigma^{(m)}(a_{\min}\lesp_n) = \sigma^{(m)}(a_{\min}\risp_n) = R_k$. Hence \eqref{eq:path} gives a path in the Rauzy graph $\Gamma_k$ which starts at $L_k$ and ends at $R_k$. If we choose $n$ such that all return words to $L_k$ are longer than $m-j$, then this path cannot visit $L_k$ twice, thus it must be exactly the inner branch. This proves the first part of the statement; the second part is dual.
\end{proof}

The next lemma is the last piece missing for the proof of the main result.
\begin{lemma}\label{l:proof-3}
    Let $\sigma$ be a primitive episturmian morphism and $u\in\lang(\sigma)$. Then:
    \begin{enumerate}
        \item for every $r\neq s\in\ret_X(u)$, $|\gcs(\sigma(rs),\sigma(sr))| \geq \ind(\sigma)$ with equality if and only if $u$ is left special;
            \label{i:proof-3}
        \item for every $r\neq s\in\bar\ret_X(u)$, $|\gcp(\sigma(rs),\sigma(sr))| \geq \frac{\Vert\sigma\Vert-|A|}{|A|-1} - \ind(\sigma)$ with equality if and only if $u$ is right special.
            \label{i:proof-3-dual}
    \end{enumerate}
\end{lemma}

\begin{proof}
    Let
    \begin{equation*}
        x = \gcs(\sigma(ab),\sigma(ba)),\quad y = \gcp(\sigma(ab),\sigma(ba)),\quad w = \pal^{-1}(xy).
    \end{equation*}
    It follows from \cref{c:gcs-gcp} that $x\sigma x^{-1} = \psi_w\circ\pi$ and $y^{-1}\sigma y = \bar\psi_w\circ\pi$, for some $\pi\in\perm(A)$. 
    
    Following \cref{n:return}, let $\dee = \dee(u)$, $\ell = \ell(u)$ and $\ell' = \ell'(u)$. Let $z$ be the prefix of length $\ell$ of $\pal(\dee)$ and $z'$ its suffix of length $\ell'$. 
    
    \ref{i:proof-3}. As a result of \cref{t:return}, we have
    \begin{equation*}
        \ret_\sigma(u) = (z^{-1}\psi_{\dee} z)(A) = z^{-1}(\psi_{\dee}(A))z.
    \end{equation*}
    Observe that $\ell=0$ if and only if $u$ is left special. Next, we compute the composition of $\sigma$ with $z^{-1}\psi_{\dee}z$:
    \begin{align*}
        \sigma\circ(z^{-1}\psi_{\dee} z) &= (x^{-1}\psi_w x) \circ \pi\circ(z^{-1}\psi_{\dee}z) \\
                               &= x^{-1}\psi_w(z)^{-1}\psi_{w\pi(\dee)}\psi_w(z)x \circ \pi \\
                               &= (\psi_w(z)x)^{-1}\psi_{w\pi(\dee)}(\psi_w(z)x) \circ \pi.
    \end{align*}
    
    In particular, it follows that $\sigma(\ret_\sigma(u)) = \rho(A)$, where $\rho$ is the episturmian morphism:
    \begin{equation*}
        \rho = (\psi_w(z)x)^{-1}\psi_{w\pi(\dee)}(\psi_w(z)x).
    \end{equation*}
    The conjugacy index of $\rho$ is $|\psi_w(z)x|$, which satisfies
    \begin{equation*}
        |\psi_w(z)x| = \ind(\sigma)+|\psi_w(z)| \geq \ind(\sigma),
    \end{equation*} 
    with equality if and only if $z=\emptyw$. As we previously observed, this is the case if and only if $u$ is left special. As $\sigma(\ret_\sigma(u)) = \rho(A)$, we have that distinct elements $r\neq s\in\sigma(\ret_\sigma(u))$ are of the form $r=\rho(a)$, $s=\rho(b)$, with $a\neq b\in A$. Finally, applying \cref{c:index}, we get
    \begin{equation*}
        |\gcs(rs,sr)| = |\gcs(\rho(ab),\rho(ba))| = \ind(\rho) = \ind(\sigma)+|\psi_w(z)| \geq \ind(\sigma),
    \end{equation*}
    with equality if and only if $u$ is left special.

    \ref{i:proof-3-dual}. Using similar arguments, we find that $\sigma(\bar\ret_\sigma(a_{\min}\lesp_n)) = \bar\rho(A)$ where
    \begin{equation*}
        \bar\rho = (y\bar\psi_w(z'))\bar\psi_{w\pi(\dee)}(y\bar\psi_w(z'))^{-1}.
    \end{equation*}
    The conjugacy index of $\bar\rho$ is given by the slightly more involved formula:
    \begin{equation*}
        \ind(\bar\rho) = \frac{\Vert\bar\rho\Vert-|A|}{|A|-1} - |y\bar\psi_w(z')| = \frac{\Vert\bar\rho\Vert-|A|}{|A|-1} - \frac{\Vert\sigma\Vert-|A|}{|A|-1} +\ind(\sigma) -|\bar\psi_w(z')|. 
    \end{equation*}
    Then for distinct elements $r,s\in \bar\ret_\sigma(u)$, we find $r = \bar\rho(a)$, $s=\bar\rho(b)$ where $a\neq b\in A$, and applying \cref{c:index}, we get
    \begin{align*}
        |\gcp(rs,sr)| &= |\gcp(\bar\rho(ab),\bar\rho(ba))| = \frac{\Vert\bar\rho\Vert-|A|}{|A|-1} - \ind(\bar\rho) \\
                      &= \frac{\Vert\sigma\Vert-|A|}{|A|-1} - \ind(\sigma) + |\bar\psi_w(z')| \geq \frac{\Vert\sigma\Vert-|A|}{|A|-1} - \ind(\sigma).
    \end{align*}
    with equality when $z' = \emptyw$, which happens precisely when $u$ is right special.
\end{proof}

We are now ready to put all the pieces together and finalize the proof of our main result.
\begin{proof}[Proof of \cref{t:main}]
    \ref{i:main-1}. Let $\sigma$ be a primitive episturmian morphism with minimal letter $a_{\min}$. Let $j = |\sigma(a_{\min})|$ and assume that $\ind(\sigma)\geq |\sigma(a_{\min})|$. Fix $n\in\nn$ such that $\lesp_n=\risp_n$. By \cref{l:proof-2}, we may assume, up to taking $n$ large enough, that $\sigma(a_{\min}\lesp_n)$ is the $(\ind(\sigma)-j)$-th inner word of its Rauzy graph. Then, choosing any two distinct return words $r\neq s\in\ret_\sigma(\sigma(a_{\min}\lesp_n))$, \cref{l:dl} gives
    \begin{equation*}
        |\gcs(rs,sr)| = \ind(\sigma)-j,
    \end{equation*}
    while by \cref{l:proof-3}, any two distinct elements $r'\neq s'\in\sigma(\ret_\sigma(a_{\min}\lesp_n))$ satisfy
    \begin{equation*}
        |\gcs(r's',s'r')| \geq \ind(\sigma) > \ind(\sigma)-j.
    \end{equation*}
    This shows that $\sigma(\ret_\sigma(a_{\min}\lesp_n)) \neq \ret_\sigma(\sigma(a_{\min}\lesp_n))$.

    \ref{i:main-2}. Now we suppose that $\ind(\sigma)\leq m-j$, where $m = \frac{\Vert\sigma\Vert-|A|}{|A|-1}$. In this case, by \cref{l:proof-2}, we may suppose, up to taking $n$ large enough, that $\sigma(\risp_na_{\min})$ is the $\ind(\sigma)$-th inner word in the Rauzy graph $\bar\Gamma_k$, where $k=|\sigma(\risp_na_{\min})|$; in fact, we have that the inner branch of $\bar\Gamma_k$ has length $m-j$. Then, by \cref{l:dl}, we have for $r\neq s\in\bar\ret_\sigma(\sigma(\risp_na_{\min}))$:
    \begin{equation*}
        |\gcp(rs,sr)| = |U_k|-\ind(\sigma) = m-j-\ind(\sigma).
    \end{equation*}
    On the other hand, by \cref{l:proof-3}, any two distinct elements $r'\neq s'\in\sigma(\bar\ret_{\sigma}(\risp_na_{\min}))$ satisfy:
    \begin{equation*}
        |\gcp(r's',s'r')| \geq m - \ind(\sigma) > m-j-\ind(\sigma).
    \end{equation*}
    Therefore we conclude that $\sigma(\bar\ret_\sigma(\risp_na_{\min})) \neq \bar\ret_\sigma(\sigma(\risp_na_{\min}))$.
\end{proof}

\section{Conclusion}\label{sec:conclu}

Let us briefly summarize the main ideas that appeared in the paper. We have established the existence of infinitely many obstructions to the preservation of return sets by episturmian morphisms. Two specific combinatorial notions were used in the description of these obstructions: the conjugacy index $\ind(\sigma)$ (\cref{d:index}) and  the minimal letter $a_{\min}$ (\cref{l:minletter}). On our way to the proof of \cref{t:main}, we established a number of properties involving these notions which are of independant interest.

The conjugacy index measures the distance to the standard episturmian morphism in its conjugacy class. According to \cref{c:index}, one observes in particular that for every $a\neq b\in A$
    \begin{align*}
        \ind(\sigma) &= \frac{\Vert\sigma\Vert - |A|}{|A|-1} - |\gcp(\sigma(ab),\sigma(ba))|\\
                     &=|\gcs(\sigma(ab),\sigma(ba))|,
        \end{align*}
with $\frac{\Vert\sigma\Vert - 1}{|A|-1} $ being equal to the cardinality of the conjugacy class of $\sigma$.

For an episturmian morphism $\sigma$ which is not a permutation, the minimal letter is defined as  the unique letter $a_{\min}$ such that $|\sigma(a_{\min})|<|\sigma(a)|$ for every $a\in A$, $a\neq a_{\min}$. In particular, when $\sigma$ is a standard morphism, i.e. when $\sigma$ is of the form $\sigma=\psi_u$ with $u \in A^+$, then $a_{\min}$ is equal to the last letter of $u$. \cref{p:minletter-epi} highlights the relationship between $\ind(\sigma)$ and $\sigma (a_{\min})$.

The proof of the main result also relies on the existence of convenient descriptions of return words and Rauzy graphs in episturmian shifts. Concerning return words, we establish in particular the following description for left and right sets of return words when $X$ is a non-degenerate episturmian shift space:  
\begin{equation*}
    \ret_X(u) = \psi_{\dee(u)}^{(\ell(u))}(A),\qquad \bar\ret_X(u) = \bar\psi_{\dee(u)}^{(\ell'(u))}(A),
\end{equation*}
where  
\begin{gather*}
     \dee(u) = \pal^{-1}(\gcs(rs,sr)\gcp(rs,sr)),\\ \ell(u) = |\gcs(rs,sr)|,\quad \ell'(u) = |\gcp(rs,sr)|,
\end{gather*}
for any choice of $r\neq s\in\ret_X(u)$ (see \cref{n:return}). 

The failure of the preservation property also reveals a \emph{shattering} of return words under the morphism. Consider as an illustration a primitive morphism $\sigma$ and a return word $r$ to the word $u$. The word $\sigma(r)$ is in general a concatenation of return words to  $\sigma(u)$. When the preservation property holds, the word $\sigma(r)$ is in fact \emph{exactly} one return word to $\sigma(u)$. On the other hand, it may happen that $\sigma(u)$ also occurs as an internal factor in $\sigma(ru)$, even if $u$ does not appear as an internal factor in $ru$. In other words, the morphism $\sigma$ can create an occurrence of $\sigma(u)$ even in the absence of  an occurrence of $u$, which causes $\sigma(r)$ to \emph{shatter}, so to speak, into several return words. While failure of the preservation property can sometimes be the result of a cardinality obstruction (as in \cite{Berthe2023}, Example 3), this cannot be the case for episturmian shift spaces, since the number of return words is constant. In particular our main result indicates that this shattering of return words occurs infinitely often.

Many questions remain open with regards to the preservation property. It is still unclear what are the precise boundaries of this property: it includes the bifix case and excludes the episturmian case, but many substitutions do not belong to either of these two families. The methods employed for constructing obstructions in the present paper relied heavily on the very special structure of episturmian morphisms. Investigating wider families (for instance dendric morphisms) would most likely require a novel approach. Another related question, which seems difficult in general, is to obtain a precise description of the set of obstructions. In the bifix case, the set of obstructions is finite and can be effectively computed; however, outside of this, we do not know how to calculate this set, even in the episturmian case.

\bibliographystyle{abbrvnat}
\bibliography{episturmian}

\vspace{1em}

\end{document}